\documentclass[10pt]{article}

\usepackage{amsmath}
\usepackage{amsfonts}
\usepackage{amssymb} 
\usepackage{amsthm}
\usepackage[affil-it,noblocks]{authblk} 
\usepackage{url}
\usepackage{bbm}
\usepackage{xcolor}
\usepackage{multicol}
\usepackage{pst-fractal}
\usepackage[numbers]{natbib}
\usepackage{comment}
\usepackage{hyperref}
\usepackage[T1]{fontenc}
\usepackage[utf8]{inputenc}
			\usepackage{stix}

\usepackage{tikz-cd}

\hypersetup{colorlinks=true,
linkcolor=blue, citecolor=blue, urlcolor=blue}
\usepackage{fullpage} 

\newtheorem{thm}{Theorem}[section]

\newtheorem{cor}[thm]{Corollary}
\newtheorem{lem}[thm]{Lemma}
\newtheorem{prop}[thm]{Proposition}

\theoremstyle{definition}
\newtheorem{defn}[thm]{Definition}

\newtheorem{rem}[thm]{Remark}

\DeclareMathOperator{\Span}{span}

\DeclareMathOperator{\MAX}{MAX}

\newcommand{\norm}[1]{\left\lVert #1 \right\rVert}
\newcommand{\abs}[1]{\left\lvert #1 \right\rvert}

\newcommand{\opsys}[3]{ ({ #1}, \{#2\}_{n \in \bb N}, {#3})}

\newcommand{\bb}[1]{\mathbb{#1}}
 
\newcommand{\Cal}[1]{\mathcal{#1}}

\newcommand{\innerproduct}[2]{\langle #1 | #2 \rangle}

\newcommand{\oline}[1]{\overline{#1}}

\newcommand{\cbnorm}[1]{\left\lVert #1 \right\rVert_{\tx{cb}}}

\renewcommand{\it}[1]{\textit{#1}}

\newcommand{\tx}[1]{\text{#1}}

\newcommand{\poright}{\preceq}

\newcommand{\wt}[1]{\widetilde{#1}}

\newcommand{\ket}[1]{\vert #1 \rangle}
\newcommand{\bra}[1]{\langle #1 \vert}
\newcommand{\kb}[2]{\ket{#1}\bra{#2}}

\newcommand{\vp}{\varphi}

\title{An abstract characterization for projections in operator systems}
\author[1]{Roy Araiza}
\author[2]{Travis Russell}
\affil[1]{Department of Mathematics, University of Illinois at Urbana-Champaign}
\affil[2]{Army Cyber Institute, United States Military Academy, West Point, NY}
\date{}

\begin{document}

\maketitle

\begin{abstract}
    We show that the set of projections in an operator system can be detected using only the abstract data of the operator system. Specifically, we show that if $p$ is a positive contraction in an operator system $\Cal V$ which satisfies certain order-theoretic conditions, then there exists a complete order embedding of $\Cal V$ into $B(H)$ mapping $p$ to a projection operator. Moreover, every abstract projection in an operator system $\Cal V$ is an honest projection in the C*-envelope of $\Cal V$. Using this characterization, we provide an abstract characterization for operator systems spanned by two commuting families of projection-valued measures. This yields a new characterization for quantum commuting correlations purely in terms of abstract operator systems.
\end{abstract}

\section{Introduction}

Beginning with the work of Choi-Effros in \cite{choi1977injectivity}, an abstract characterization for self-adjoint unital subspaces of the bounded operators on a Hilbert space was given. More recently the abstract theory of \it{operator systems} progressed further with the development of the theory of tensors. In particular it was shown in \cite{kavruk2011tensor, kavruk2013quotients} that if classes of operator systems satisfied certain nuclearity properites then it must follow that $C^*(F_\infty)$ had Lance's weak expectation property, i.e. particular nuclearity properties of operator systems were proven to be equivalent to Kirchberg's conjecture. Kirchberg showed in \cite{kirchberg1993non} that if the local lifting property for C*-algebras implied the weak expectation property then Connes' embedding conjecture, originally appearing in \cite{ConnesConjecture}, must hold. In \cite{junge2011connes}, \cite{fritz2012tsirelson}, and \cite{ozawa2013connes}, an equivalence between Kirchberg's conjecture and what is known as Tsirelson's problem was established. Tsirelson's problem asks if for all pairs of natural numbers $(n,k)$ the equality $C_{qa}(n,k) = C_{qc}(n,k)$ holds, where $C_{qa}(n,k)$ denotes the closure of the set of \it{quantum correlations}, and $C_{qc}(n,k)$ denotes the set of \it{quantum commuting} correlations. Introducing new ideas coming from computer science, a recent preprint \cite{ji2020mip} demonstrates the existence of integers $n$ and $k$ such that $C_{qa}(n,k) \neq C_{qc}(n,k)$, simultaneously refuting the long-standing conjectures of Tsirelson, Kirchberg, and Connes. Sharp estimates on the ordered pairs $(n,k)$ for which $C_{qa}(n,k) \neq C_{qc}(n,k)$ are not known. Estimating these values could shed more light on the failure of the conjectures of Kirchberg and Connes, which would be useful in the study of operator algebras. The purpose of this paper is to better understand the role played by projections in operator systems, since projections play an outsized role in Tsirelson's problem. Our principle motivation is the hope that new insights about the structure of operator systems may be useful in the study of Tsirelson's problem.

\indent In this paper, we provide an abstract characterization for the set of projections in an operator system. Given an abstract operator system $\Cal V$ and a positive element $p \in \Cal V$ of unit norm (as induced by the order unit and positive cone on $\Cal V$), we consider a collection of cones $\{C(p_n)\}_n$ induced by $p$ and prove that the quotient $*$-vector space $\Cal V/ J_p$ is an operator system with order unit $p + J_p$ and matrix ordering $\{C(p_n) + M_n(J_p)\}_n$, where $J_p = \Span C(p) \cap -C(p)$ (Theorem \ref{thm: abstract compression}). When $\Cal V$ is a concrete operator system in $B(H)$ and $p$ is a projection, we show that $\Cal V / J_p$ is completely order isomorphic to the compression operator system $p \Cal V p \subseteq B(pH)$ (Corollary \ref{cor: concrete and abstract compression systems agree}). We call a positive element $p \in \Cal V$ an \textit{abstract projection} if $p$ has unit norm and the mapping  \begin{align}
    \pi_p: \Cal V \to M_2(\Cal V)/J_{p \oplus q}, \quad x \mapsto \begin{pmatrix}
    x & x \\
    x & x
    \end{pmatrix} + J_{p \oplus q} \nonumber
\end{align} is a complete order isomorphism, where $q = e - p$ and $e$ is the unit of $\Cal V$ (Definition \ref{defn: abstract projection}). Our main result is that a positive element $p \in \Cal V$ is an abstract projection if and only if there exists a Hilbert space $H$ and a unital complete order embedding $\pi: \Cal V \to B(H)$ such that $\pi(p)$ is a projection (Theorem \ref{thm: projection representation}), establishing a one-to-one correspondence between abstract and concrete projections. Similar to a result of Blecher and Neal appearing in \cite[Lemma 2.3]{blecher2011metric}, we then prove in Theorem \ref{thm: projections in C*-envelope} that $p \in \Cal V$ is an abstract projection if and only if $p$ is a projection in the C*-envelope $C_e^*(\Cal V).$ These observations lead us quickly to a new characterization of the set of quantum commuting correlations $C_{qc}(n,k)$ in terms of abstract operator systems and abstract projections. Specifically, we show that $\{p(a,b|x,y)\} \in C_{qc}(n,k)$ if and only if $p(a,b|x,y) = \phi(Q(a,b|x,y))$ where $\phi$ is a state on an operator system $\Cal V$ spanned by abstract projections $\{Q(a,b|x,y)\}$ satisfying certain operator non-signalling conditions, namely that \[ \sum_{a,b} Q(a,b|x,y) = e \] for all $x,y$ (where $e$ is the unit of $\Cal V$) and that the marginal operators \[ E(a|x) := \sum_b Q(a,b|x,y) \quad \text{and} \quad F(b|y) := \sum_b Q(a,b|x,y) \] are well defined (Theorem \ref{thm: characterization correlations}).

\indent We conclude this introduction by pointing out some related results in the literature. The idea of abstractly characterizing certain types of operators in operator spaces was investigated extensively by Blecher and Neal \cite{blecher2011metric, blecher2013metric}, who studied the abstract (linear-metric) structure of an operator space that contains a unit. More precisely, given a pair $(\Cal E, u)$ where $\Cal E$ is an \it{abstract operator space} and $u \in \Cal E$, they characterized when there exists a completely isometric embedding of $\Cal E$ into $B(H)$ such that $u$ is mapped to a unitary. They also showed that given a unitary $u \in \Cal E$, then $u$ is necessarily a unitary in the ternary envelope $T(\Cal E)$ of $\Cal E$. Other operator system characterizations for correlation sets can be found in the literature, for example in Theorem 3.1 of \cite{lupini2020perfect} and Theorem 2.4 of \cite{Harris2017UnitaryCS}. These papers characterize correlations as having the form $p(a,b|x,y) = \phi(E_{x,a} \otimes E_{y,b})$ where $\{E_{x,a}\}_a$ are projection-valued measures spanning a certain canonical operator system $S$ and $\phi$ is a state on $S \otimes_t S$, where $\otimes_t$ denotes the various operator system tensor products of \cite{kavruk2011tensor} depending on which type of correlation one intends to construct. Our results differ in that we do not appeal to the hierarchy of operator system tensor products or make use of any canonical operator system. We do not, however, have any result analogous to theirs for quantum approximate or local correlations.

Our paper is organized as follows. In Section \ref{sec: preliminaries}, we provide some elementary background in operator systems and operator spaces, including some preliminary remarks that will be useful throughout the paper. In Section \ref{sec: concrete compression operator systems} we study the compression of a concrete operator system $\Cal V$ by a projection $p \in \Cal V$, i.e., the operator system $p \Cal V p \subseteq B(pH)$. The observations of this section motivate the results of Section \ref{sec: abstract compression operator systems}, where we provide an abstract definition for the compression of an abstract operator system by a positive contraction $p$ which is not a priori a projection. In Section \ref{sec: projections in operator systems}, we combine the results of Section \ref{sec: concrete compression operator systems} and Section \ref{sec: abstract compression operator systems} and some additional observations to prove the main results of the paper. We conclude with Section \ref{sec: applications to qit} where we present our applications to the theory of quantum correlation sets.

\section{Preliminaries}\label{sec: preliminaries}

 Though we assume some familiarity with general operator system and operator space theory we will review some definitions and constructions that appear throughout the manuscript. 
 
 \begin{defn}
 Given a Hilbert space $H$ a \it{concrete operator system} is a self-adjoint unital subspace $\Cal V$ of $B(H)$. An \it{abstract operator system} is defined to be the triple $\opsys{\Cal V}{C_n}{e}$ where $\Cal V$ is a complex vector space with a conjugate-linear involution $*$ (i.e. a \textit{$*$-vector space}), $\{C_n\}_n$ is a proper matrix ordering on $\Cal V$ and $e$ is an Archimedean matrix order unit. By a matrix ordering on $\Cal V$, we mean a sequence $\{C_n\}_n$ which satisfies the following two properties:
 \begin{enumerate}
     \item $C_n \subset M_n(\Cal V)_h$ is a cone whose elements are invariant under the involution $*$; 
     \item Given any $n,m \in \bb N$ and $\alpha \in M_{n,m}$ then $\alpha^*C_n\alpha \subset C_m.$ 
 \end{enumerate} A matrix ordering $\{C_n\}_n$ is called \textit{proper} if in additon $C_n \cap -C_n = \{0\}$ for each $n$. A element $e \in \Cal V_h$ is called an \it{Archimedean matrix order unit} for a matrix ordering $\{C_n\}_n$ if given any $x \in M_n(\Cal V)_h$ there exists $r>0$ such that if $e_n:= I_n \otimes e,$ then $re_n - x \in M_n(\Cal V)^+$ and if $re_n + x \in M_n(\Cal V)^+$ for all $r>0$ then $x \in M_n(\Cal V)^+.$ 
 \end{defn}
 When dealing with ``compression'' operator systems beginning in Sections \ref{sec: concrete compression operator systems} we will be dealing with cones that a priori are not necessarily proper. When it may be unclear from context, we will emphasize when a matrix ordering is not necessarily proper, and otherwise simply use the term \textit{matrix ordering}. Our use of the term matrix ordering differs from some authors (e.g. \cite{kavruk2011tensor}) where the term matrix ordering is used synonymously with proper matrix ordering. We note that the family $\{C_n\}_n$ with $C_n \subset M_n(\Cal V)_h$ is a matrix ordering if $C_n \oplus C_m \subset C_{m+n}$ and $\alpha^*C_n\alpha \subset C_m$ for all $\alpha \in M_{n,m}, n,m \in \bb N.$

 When no confusion will arise we will simply denote an operator system by $\Cal V$. The morphisms in use between matrix ordered $*$-vector spaces will be completely positive maps, and the morphisms between operator systems will be the unital completely positive maps. Given a linear map $\vp: \Cal V \to \Cal W$ between operator systems, then for each $n \in \bb N$ there is an induced linear map $I_n \otimes \vp:= \vp_n: M_n(\Cal V) \to M_n(\Cal W)$ defined by $\sum_{i,j \leq n} \kb{i}{j} \otimes x_{ij} \mapsto \sum_{i,j \leq n} \kb{i}{j} \otimes \vp(x_{ij}).$ The map $\vp_n$ is called the \it{nth amplification} of $\vp.$ Here we have let $\{\ket{i}\}_i$ denote column vectors in $\bb C^n$ with $1$ in the $i$th position. We say that $\vp$ is \it{completely positive} if $\vp_n(M_n(\Cal V)^+) \subset M_n(\Cal W)^+$ for every $n$, and is a \it{complete order isomorphism} if $\vp$ is invertible with $\vp$ and $\vp^{-1}$ both completely positive. A map $\vp: \Cal V \to \Cal W$ which is not necessarily surjective is called a \textit{complete order embedding} if it is completely positive, injective and $\vp^{-1}$ is completely positive on $\vp(\Cal V) \subset \Cal W$. We will identify two operator systems $\Cal V$ and $\Cal W$ if there exists a (unital) complete order isomorphism $\vp$ between the two and we will denote this by $\Cal V \simeq \Cal W.$ A classical result due to Choi and Effros shows that there is a one-to-one correspondence between concrete and abstract operator systems. 
 
 \begin{thm}[\cite{choi1977injectivity}]
 Given an abstract operator system $\Cal V$ then there is a Hilbert space $H$ and a concrete operator system $\Cal W \subset B(H)$ such that $\Cal V \simeq \Cal W$. Conversely, every concrete operator system is an abstract operator system.
 \end{thm}
 
 A well-known fact is that given a matrix-ordered $*$-vector space, i.e., a pair $(\Cal V, \{C_n\}_n)$ consisting of a $*$-vector space $\Cal V$ and a matrix-ordering $\{C_n\}_n$, then an element $e \in \Cal V^+$ is an order unit for $\Cal V$ if and only if it is a matrix order unit. We provide a brief proof below for completeness.

\begin{lem}[\cite{paulsen2011operator}]
Given an $*$-vector space $\Cal V$ then $M_n(\Cal V)_h = (M_n)_h \otimes \Cal V_h$.
\end{lem}

\begin{prop}\label{prop: order unit iff matrix order unit}
Given a matrix ordered $*$-vector space $\Cal V$ then $e \in \Cal V$ is an order unit if and only if it is a matrix order unit.
\end{prop}
\begin{proof}
Of course we need only show the forward direction. Let $x \in M_n(\Cal V)_h$ such that $x = \sum_{i \leq n} A_i \otimes x_i \in (M_n)_h \otimes \Cal V_h$. For each $i$ write $A_i = P_i - Q_i$, where $P_i, Q_i \in M_n^+.$ Choose $\lambda > 0$ such that $\lambda e \pm x_i \in \Cal V^+$ for each $i$. We then see \begin{align*}
    \lambda(\sum_{i \leq n} P_i + Q_i)\otimes e - x = \sum_{i \leq n} P_i \otimes (\lambda e -x_i) + \sum_{i \leq n} Q_i \otimes (\lambda e + x_i) \in M_n(\Cal V)^+.
\end{align*} Simply choose $\wt{\lambda}$ such that $\wt{\lambda} I_n \geq \lambda (\sum_{i \leq n} P_i + Q_i)$ which proves the claim.
\end{proof}

\begin{rem}[The Canonical Shuffle]\label{rem: the canonical shuffle}
Throughout the manuscript we will implement the ``canonical shuffle'' (see \cite[Chapter 8]{paulsen2002completely}). For example, given an operator system $\Cal V \subset B(H)$, we use this in Lemma \ref{lem: the nth amplification of the compression (p,q) is (p_n,q_n)} to identify $M_n(\Cal V \oplus \Cal V)$ with $M_n(\Cal V) \oplus M_n(\Cal V),$ where the direct sum is in the $\ell_\infty$ sense, i.e., we are realizing $\Cal V \oplus \Cal V$ as the diagonal $2 \times 2$ matrices with entries from $\Cal V$. In particular, consider $A \in M_n(M_m(\Cal V))$. We then write \begin{align}
A = \sum_{ij}\ket{i}\bra{j} \otimes A_{ij}, A_{ij} \in M_m(\Cal V).
\end{align} It follows \begin{align}
    A = \sum_{ij}\ket{i}\bra{j} \otimes A_{ij} = \sum_{ij} \ket{i}\bra{j} \otimes \sum_{kl} \ket{k}\bra{j} \otimes a_{ijkl}, a_{ijkl} \in \Cal V.
\end{align} Here $\{\kb{i}{j}\}_{i,j \leq n}$ denotes the matrix units of $M_n$ and $\{\kb{k}{l}\}_{k,l \leq m}$ denotes the matrix units of $M_m.$ We see \begin{align}
    \sum_{ij}\ket{i}\bra{j} \otimes A_{ij} = \sum_{ij} \ket{i}\bra{j} \otimes \sum_{kl} \ket{k}\bra{l} \otimes a_{ijkl} = \sum_{kl} \ket{k}\bra{l} \otimes \sum_{ij} \ket{i}\bra{j} \otimes a_{ijkl} = \sum_{kl} \ket{k}\bra{l} \otimes B_{kl} = B,
\end{align} where $B_{kl} = \sum_{ij} \ket{i}\bra{j} \otimes a_{ijkl},$ and $B \in M_m(M_n(\Cal V)).$ This canonical map is a $*$-isomorphism between the ambient $C^*$-algebras, i.e., betweeen $M_n(M_m(B(H)))$ and $M_m(M_n(B(H)))$. One may also view this using the commutativity of the operator system minimal tensor product, i.e., $M_n(M_m(\Cal V)) \simeq M_n \otimes_{\tx{min}} M_m \otimes_{\tx{min}}  \Cal V \simeq M_m \otimes_{\tx{min}}  M_n \otimes_{\tx{min}}  \Cal V \simeq M_m(M_n(\Cal V)).$
\end{rem}

Given an operator system $\Cal V$ consider now a pair $(\kappa, \Cal A)$ where $\kappa: \Cal V \to \Cal  A$ is a unital complete order embedding, and $\Cal A$ is a C*-algebra. We will call the pair a \it{C*-extension} of $\Cal V$ if $\Cal A$ is generated by the image $\kappa(\Cal V)$ as a C*-algebra. In particular, there exists a minimal such extension satisfying a universal property. Given an operator system $\Cal V$ and two C*-extensions, $(\kappa_1, \Cal  A_1), (\kappa_2, \Cal A_2)$, we say that the extensions are \it{$\Cal V$-equivalent} if there exists a $*$-isomorphism $\pi: \Cal A_1 \to \Cal A_2$ such that $\pi \kappa_1 = \kappa_2.$ 

\begin{thm}[Arveson-Hamana]\label{thm: C*-envelope existence}
Given an operator system $\Cal V$ then there exists a C*-extension $(\kappa,A)$ satisfying the following universal property: given any other C*-extension $(j,B)$ of $\Cal V$ then there exists a unique $*$-epimorphism $\pi: B \to A$ such that $\pi \circ j = \kappa.$
\end{thm}

\begin{defn}
Given an operator system $\Cal V$ then the \it{C*-envelope} of $\Cal V$ will be any C*-extension satisfying Theorem \ref{thm: C*-envelope existence}. Such a C*-extension is unique up to $\Cal V$-equivalence and we denote it by $C_{\tx{e}}^*(\Cal V).$ 
\end{defn}

Given a Hilbert space $H$ a \it{concrete operator space} is a closed subspace $\Cal E \subset B(H)$. An \it{abstract operator space} will be the pair $(\Cal E, \{\alpha_n\}_n)$ where $\Cal E$ is a linear space and $\{\alpha_n\}_n$ is a sequence of matrix norms on $\Cal E$, that is $\alpha_n: M_n(\Cal E) \to [0,\infty)$ for all $n \in \bb N$, satisfying Ruan's axioms. This is to say that the following two properties are satisfied: \begin{enumerate}
    \item $\alpha_{m+n}(x \oplus y) = \max\{\alpha_m(x), \alpha_n(y)\}$ for all $x \in M_m(\Cal E), y \in M_n(\Cal E)$;
    \item $\alpha_m(axb) \leq \norm{a} \alpha_m(x) \norm{b}$ for all $a,b \in M_m.$
\end{enumerate} When no confusion will arise we will simply denote an operator space by $\Cal E.$ Given a linear map $\vp: \Cal E \to \Cal F$ between operator spaces, then we say $\vp$ is \it{completely bounded} if $\cbnorm{\vp}:= \sup_n \norm{\vp_n} < \infty$, and $\vp$ will be called \it{completely isometric} if $\vp_n$ is an isometry for all $n \in \bb N$. We identify two operator spaces if there exists a completely isometric bijection between the two. This is to say that $\Cal E$ is completely isometric to $\Cal F$ if there exists a linear map $\vp: \Cal E \to \Cal F$ which is invertible and both $\vp$ and $\vp^{-1}$ are completely contractive. We will thus say that $\Cal E$ is \it{completely isometric} to $\Cal F.$ Due to a result of Ruan there is a one-to-one correspondence between concrete and abstract operator spaces. 

\begin{thm}[\cite{ruan1987matricially}]
Given an abstract operator space $\Cal E$ then there exists a Hilbert space $H$ and a concrete operator space $\Cal F \subset B(H)$ such that $\Cal E$ and $\Cal F$ are completely isometric. Conversely, any concrete operator space is an abstract operator space.
\end{thm}

It is natural to consider the structural properties of an operator space that contains a ``unit''.
\begin{defn}[\cite{blecher2011metric}]
A \it{unital operator space} is a pair $(\Cal E, u)$ where $\Cal E$ is an operator space and $u \in \Cal E$ is such that there exists a completely isometric embedding $\vp: \Cal E \to B(H)$ such that $\vp(u) = I_H.$ We will call $u$ a \it{unitary} of $\Cal E$ if there exists a completely isometric embedding which maps $u$ to a unitary of some $B(H)$. 
\end{defn} 

Analogous to the discussion preceding Theorem \ref{thm: C*-envelope existence} and the theorem itself, we may talk about ternary extensions of unital operator spaces. Consider a ternary ring of operators (TRO) $Z \subset B(K,H)$. This is to say that $Z$ is a closed subspace of $B(K,H)$ and $ZZ^*Z \subset Z.$ Then an element $u \in Z$ is called a \it{C*-unitary} if for all $z \in Z$ we have $zu^*u = z$ and $uu^*z = z.$ Given two TROs $Z$ and $W$, a \textit{ternary morphism} is a linear map $\phi:Z \to W$ such that $\phi(xy^*z) = \phi(x)\phi(y)^*\phi(z)$ for each $x,y,z \in Z$. Given an operator space $\Cal E$, a \textit{ternary envelope} for $\Cal E$ is a TRO $T(\Cal E)$ generated by $\Cal E$ satisfying the following universal property: if $Z$ is some other ternary envelope generated by $\Cal E$, then there exists a ternary morphism $\pi: Z \to T(\Cal E)$ extending the identity map on $\Cal E$. The existence of the ternary envelope was established in \cite{HamanaTripleEnvelope}.

\begin{thm}[\cite{blecher2011metric}] \label{thm: Blecher-Neal}
Given an operator space $\Cal E$. the following are equivalent for an element $u \in \Cal E$: \begin{enumerate}
    \item $u$ is a unitary in $\Cal E$.
    \item There exists a TRO $Z$ containing $\Cal E$ completely isometrically such that $u$ is a $C^*$-unitary in $Z$. 
    \item $u$ is a $C^*$-unitary in the ternary envelope $T(\Cal E).$
\end{enumerate}
\end{thm}

\begin{rem}
Throughout the rest of the manuscript there will be results where we will be considering $\Cal V$ as already sitting in some $B(H)$. In particular, we will consider projections in a concrete operator system in Section \ref{sec: concrete compression operator systems} and projections in an abstract operator system in Section \ref{sec: abstract compression operator systems}. Thus we will interchange between the use of $e$ or $I$ as the Archimedean order unit dependent on whether our operator system is abstract or concrete. The Archimedean order unit will always be clear from context. 
\end{rem}

\section{Concrete compression operator systems}\label{sec: concrete compression operator systems}

In this section we present the motivation for Sections \ref{sec: abstract compression operator systems} and \ref{sec: projections in operator systems}. Starting with a concrete operator system $\Cal V \subset B(H)$ we show that an element $p \in \Cal V^+$, which is also a projection on $B(H)$ induces a natural collection of cones which correspond to the hermitian elements of $\Cal V$ whose compression by $p$ is positive. These cones form a matrix ordering and in particular, will form a proper matrix ordering on a certain quotient $*$-vector space. It will follow that image of $p$ in the quotient is an Archimedean matrix order unit for the space. 

Let $\Cal V \subset B(H)$ be an operator system and let $p \in \Cal V$ be a projection (as an operator in $B(H)$). Then letting $p_n:= I_n \otimes p$ we consider the collection of sets $\{C(p_n)\}_n$ where for each $n$ \begin{align*}
    C(p_n) = \{x \in M_n(\Cal V)_h: p_nxp_n \in B(H^{n})^+\},
\end{align*}where $H^{n}$ denotes the $n$-fold Hilbertian direct sum. We will show that the sequence $\{C(p_n)\}_n$ is a (not necessarily proper) matrix ordering on $\Cal V$.

\begin{prop}\label{prop: Properties of C(p)} Let $\Cal V \subset B(H)$ be an operator system and suppose that $p \in \Cal V$ where $p$ is a projection in $B(H)$. The sequence of sets $\{C(p_n)\}_n$ is a matrix ordering on $\Cal V.$ Furthermore if $p \leq q \leq I$ then $q$ is an Archimedean matrix order unit for $(\Cal V, \{C(p_n)\}_n).$
\end{prop}

\begin{proof}
By definition $C(p_n)^* = C(p_n)$ for all $n \in \bb N.$ The compression by the projection $p$ is a linear map and thus $\lambda C(p_n) \subset C(p_n)$ for all $n \in \bb N, \lambda >0$ and $C(p_n) + C(p_n) \subset C(p_n).$ Finally, let $\alpha \in M_{n,m}, x \in C(p_n).$ Then \begin{align*}
p_m\alpha^*x\alpha p_m = p_m\left( \left[ \sum_{kl} \overline{\alpha}_{ki}x_{kl}\alpha_{lj} \right]_{ij} \right)p_m = \left[ p\left(\sum_{kl} \overline{\alpha}_{ki}x_{kl}\alpha_{lj}\right)p \right]_{ij} = \left[ \sum_{kl}\overline{\alpha}_{ki}px_{kl}p\alpha_{lj} \right]_{ij} = \alpha^*(p_nxp_n)\alpha,
\end{align*} and $\alpha^*(p_nxp_n)\alpha \in B(H^m)^+.$ This proves the first statement.

We now show that if $p \leq q \leq I$ then $q$ is an Archimedean matrix order unit for the pair $(\Cal V, \{C(p_n)\}_n).$ We first make some observations. It is immediate that given the projection $p \in \Cal V$ then $p$ is an Archimedean matrix order unit for operators of the form $pzp, z \in \Cal V.$ Simply notice that for any $n \in \bb N$ and $x \in M_n(\Cal V)$, then $I_n \geq p_n$ and if for $\epsilon >0$ we have $\epsilon p_n + p_nxp_n \in B(H^n)^+$ then \begin{align*}
    \epsilon I_n + p_nxp_n \geq \epsilon p_n + p_nxp_n \in B(H^n)^+.
\end{align*} Since $I$ is an Archimedean matrix order unit for the operator system $B(H)$ we have $p_nxp_n \in B(H^n)^+$ and thus $x \in C(p_n).$

By the assumption that $p \leq q $ it follows that $p \leq pqp$. (In fact since $q \leq I$ we have $pqp = p$ since it also follows $p - pqp = p(I-q)p \geq 0$). Let $x = x^* \in \Cal V$ and let $r>0$ such that $rp - pxp \geq  0.$ Then \begin{align}
p(rq -x)p = rpqp - pxp \geq rp - pxp \geq 0.
\end{align} Thus $rq - x \in C(p)$ implying $q$ is an order unit. If for all $\epsilon >0, \epsilon q + x \in C(p)$ then $\epsilon pqp + pxp \geq 0$ for all $\epsilon >0$ and thus \begin{align*}
\epsilon p + pxp \geq \epsilon pqp + pxp \geq 0, \forall \epsilon >0,
\end{align*} implying $pxp \geq 0$ ( and thus $x \in C(p)$) by our earlier remarks on $p$, and thus $q$ is an Archimedean order unit. 

The same observations hold for $p_n$ and $q_n$ for all $n \in \bb N$ and thus the same method will show that $q$ is an Archimedean matrix order unit for $(\Cal V, \{C(p_n)\}_n)$.

\end{proof}

The next corollary, which follows from Proposition \ref{prop: Properties of C(p)}, states that the cones $C(p_n)$ are closed in the order seminorm induced by the projection $p$. 

\begin{cor}
Given an operator system $\Cal V \subset B(H)$ and projection $p \in \Cal V$ then if $\{C(p_n)\}_n$ denotes the matrix ordering as in Lemma \ref{prop: Properties of C(p)} then for each $n \in \bb N$, $C(p_n)$ is closed in the order seminorm. 
\end{cor}

\begin{proof}
For this proof let $\alpha_n: M_n(\Cal V) \to [0,\infty)$ denote the order seminorm defined for each $n \in \bb N$ by \[ \alpha_n(x):= \inf \{ r >0: \begin{pmatrix}
rp_n & x \\
x^* & rp_n
\end{pmatrix} \in C(p_{2n})\}. \] Let $\{x_i\}_{i \in I} \subset C(p_n)$ denote a net such that $x = \alpha_n$-$\lim_i x_i$ ($x$ is the limit of $x_i$ relative to the norm $\alpha_n$). $x$ is necessarily $*$-hermitian. Let $r >0$ and let $i_o \in I$ such that $\alpha_n(x_i -x) < r$ for $i_o \poright i.$ It follows that \[ \begin{pmatrix}
rp_n & x - x_i \\
x - x_i & rp_n
\end{pmatrix} \in C(p_{2n}) \] and thus compression by $\begin{pmatrix}
1 \\
1
\end{pmatrix}$ implies $rp_n + x -x_i \in C(p_n)$ and therefore $rp_n + x \in C(p_n)$. Since $p$ is an Archimedean matrix order unit we have $x \in C(p_n)$ as desired.
\end{proof}

\begin{rem}
Used in Proposition \ref{prop: Properties of C(p)} we make the point of saying that if $\vp: \Cal V \to B(H)$ is the map defined by $\vp(v) = pvp$ then $\vp_n(v) = p_nvp_n,v \in M_n(\Cal V).$ First note that compression by the projection is linear completely positive. Notice, \begin{align*}
    p_nvp_n = (I_n \otimes p)(\sum_{ij}\ket{i}\bra{j} \otimes v_{ij})(I_n \otimes p ) = \sum_{ij} \ket{i}\bra{j} \otimes pv_{ij}p = \sum_{ij}\ket{i}\bra{j} \otimes \vp(v_{ij}) = \vp_n(v),
\end{align*}
\end{rem}

\begin{lem}\label{lem: M_n(J_p) = J_(p_n) }
Given an operator system $\Cal V \subset B(H)$ let $p \in \Cal V$ be a projection as above. Let $J_p = \Span C(p) \cap -C(p),$ and $J_{p_n} = \Span C(p_n) \cap -C(p_n).$ Then $M_n(J_p) = J_{p_n}$.
\end{lem}
\begin{proof}
If $x \in M_n(J_p)$ then we write $x = \sum_{ij} \ket{i}\bra{j} \otimes x_{ij}, x_{ij} \in J_p.$ For each $i,j$ we then write $x_{ij} = \sum_k \lambda_{ijk} a_{ijk}$ where $\lambda_{ijk} \in \bb C, a_{ijk} \in C(p) \cap -C(p).$ We then see \begin{align*}
    p_nxp_n = p_n(\sum_{ijk} \lambda_{ijk} \ket{i}\bra{j} \otimes a_{ijk})p_n = \sum_{ijk} \lambda_{ijk} \ket{i}\bra{j} \otimes pa_{ijk}p = 0,
\end{align*}since $\pm pa_{ijk}p \in B(H)^+$ for all $i,j,k$ and therefore is equal to 0. Thus $x \in C(p_n) \cap -C(p_n)$ which yields the first inclusion. Conversely, let $x \in J_{p_n}$ and write $x = \sum_k \lambda_k x_k, \lambda_k \in \bb C, x_k \in C(p_n) \cap -C(p_n).$ We write $$x_k = \sum_{ij} \ket{i}\bra{j} \otimes a_{ijk} \in M_n(\Cal V).$$ By the assumption that $p_nx_kp_n = 0$ we see that \begin{align*}
    p_nx_kp_n = \sum_{ij} \ket{i}\bra{j} \otimes pa_{ijk}p = 0,
\end{align*} for all $i,j,k,$ which in turn yields that $pa_{ijk}p = 0$ and thus $a_{ijk} \in C(p) \cap -C(p).$ This then yields \begin{align*}
    x = \sum_{ijk} \ket{i}\bra{j} \otimes \lambda_ka_{ijk} \in M_n(J_p),
\end{align*} finishing the proof. 
\end{proof}

\begin{lem}\label{lem: compatibility of J_p with operator system structure}
Given an operator system $\Cal V \subset B(H)$ with projection $p \in \Cal V$, then for all $n \in \bb N$, $J_{p_n} \subset M_n(\Cal V)$ is $*$-closed and if $\alpha \in M_{n,m}$ then $\alpha^*J_{p_n}\alpha \subset J_{p_m}.$ In particular $p_nJ_{p_n}p_n \subset J_{p_n}.$
\end{lem}
\begin{proof}
Let $x \in J_{p_n}$. Then $x = \sum_{ijk} \ket{i}\bra{j} \otimes \lambda_k a_{ijk}$ with, as we saw in the proof of Lemma \ref{lem: M_n(J_p) = J_(p_n) }, $a_{ijk} \in C(p) \cap -C(p)$ for all $i,j,k,$ and therefore $a_{ijk}^*= a_{ijk}.$ Thus $x^* = \sum_{ijk} \ket{j}\bra{i} \otimes \oline{\lambda}_ka_{ijk} \in J_{p_n}.$ It is immediate that for any $x \in J_{p_n}$ that $p_nxp_n \in J_{p_n}$ since if we write $x = \sum_k \lambda_k x_k,$ with $x_k \in C(p_n) \cap -C(p_n)$ then $p(px_kp)p  = px_kp \in B(H^n)^+.$ This holds for all $k$ and thus $p_nJ_{p_n}p_n \subset J_{p_n}.$

\indent Finally if $x \in J_{p_n}$ and $\alpha \in M_{n,m}$ we have \begin{align*}
    \pm p_m\alpha^*x\alpha p_m = \alpha^*(\pm p_nxp_n)\alpha \in B(H^m)^+.
\end{align*} Thus, $\alpha^*x\alpha \in J_{p_m}.$ 
\end{proof}

We now wish to consider the vector space $\Cal V / J_p$ where $\Cal V \subset B(H)$ is an operator system, and $p \in \Cal V$ is a projection in $B(H)$. The vector space operations and involution are defined in the natural way and are well-defined by Lemma \ref{lem: compatibility of J_p with operator system structure}. We consider the family $\{\wt{C}(p_n)\}_n$ where for each $n \in \bb N$ we have \begin{align}
    \wt{C}(p_n):= \{ (x_{ij} + J_p) \in M_n(\Cal V/ J_p): x = (x_{ij}) \in C(p_n)\}.
\end{align}

Note that $p \notin J_p,$ for $p$ a nonzero projection in $\Cal V$. This is immediate since if $p = \sum_k \lambda_k c_k, \pm c_k \in C(p) \cap -C(p)$ then it would follow that \begin{align}
    p = ppp = p(\sum_k\lambda_kc_k)p = \sum_k \lambda_k pc_kp,
\end{align} and since $\pm pc_kp \in B(H)^+$ for all $k$ implies $pc_kp = 0$. Thus, $p = 0.$  

\begin{thm}\label{thm: the compression of an operator system is an operator system} The triple $(\Cal V/ J_p,\{\wt{C}(p_n)\}_n,p+ J_p)$ is an operator system.
\end{thm}

\begin{proof}
We begin by showing that $\{\wt{C}(p_n)\}_n$ is a proper matrix ordering. Let $(x_{ij} + J_p) \in M_n(\Cal V/ J_p).$ If $\lambda >0$ then $\lambda (x_{ij} + J_p) = (\lambda x_{ij} + J_p)$ and $\lambda x \in \lambda C(p_n) \subset C(p_n)$, thus $\lambda (x_{ij} + J_p) \in \wt{C}(p_n).$ If $(x_{ij} + J_p), (y_{ij} + J_p) \in \wt{C}(p_n)$ then \begin{align*}
    (x_{ij} + J_p) + (y_{ij} + J_p) = (x_{ij} + y_{ij}) + J_p \in \wt{C}(p_n)
\end{align*} since $x+ y \in C(p_n).$ If $m \in \bb N$ and $(x_{ij} + J_p) \in \wt{C}(p_n)$ and $\alpha \in M_{n,m}$ then \begin{align*}
    \alpha^*(x_{ij} + J_p)\alpha = \sum_{ijkl} \ket{i}\bra{j} \otimes (\oline{\alpha}_{ki}x_{kl}\alpha_{lj} + J_p) = (\sum_{kl}\oline{\alpha}_{ki}x_{kl}\alpha_{lj} + J_p)_{ij},
\end{align*} and $\alpha^*x\alpha \in C(p_m)$ which implies $\alpha^*(x_{ij} + J_p)\alpha \in \wt{C}(p_m).$ If $(x_{ij} + J_p) \in \wt{C}(p_n)$ then $(x_{ij} + J_p)^* = (x_{ji}^*+ J_p)$ and since $x \in C(p_n)$ implies $x^* \in C(p_n)$ which implies $(x_{ij} + J_p)^* = (x_{ji}^*+ J_p) \in \wt{C}(p_n).$ Finally, suppose that $(x_{ij} + J_p) \in \wt{C}(p_n) \cap - \wt{C}(p_n).$ This implies that $\pm x \in C(p_n).$ Then it follows $x \in J_{p_n} = M_n(J_p)$ by Lemma \ref{lem: M_n(J_p) = J_(p_n) } and thus $x_{ij} \in J_p$ for all $i,j.$ In other words $(x_{ij} + J_p) = (0 + J_p)$ implying that our cones are proper, and thus the matrix ordering $\{\wt{C}(p_n)\}_n$ on $\Cal V/ J_p$ is proper. It remains to show that $p + J_p$ is an Archimedean matrix order unit. Consider a $*$-hermitian element $(x_{ij} + J_p) \in M_n(\Cal V/ J_p).$ Let $r>0$ be such that $rp_n -x \in C(p_n)$ (see Proposition \ref{prop: Properties of C(p)}). It then follows \begin{align*}
    r(I_n \otimes (p + J_p)) - (x_{ij} + J_p) = (I_n \otimes (rp + J_p)) -(x_{ij} + J_p) = ((r(p_n)_{ij} - x_{ij}) + J_p) \in \wt{C}(p_n),
\end{align*} since $rp_n - x \in C(p_n).$

Finally if for all $r>0$ we have $(r(p_n)_{ij} + x_{ij} + J_p) \in \wt{C}(p_n)$ then $rp_n + x \in C(p_n)$ which implies $x \in C(p_n)$ (see Proposition \ref{prop: Properties of C(p)}) and thus $(x_{ij} + J_p) \in \wt{C}(p_n)$. This finishes our proof.

\end{proof}

\begin{defn}\label{defn: concrete compression operator system}
Given an operator system $\Cal V \subset B(H)$ with $p \in B(H)$ a projection, we call the set $p\Cal V p$, regarded as linear operators on the Hilbert space $pH$, the \it{concrete compression operator system}.
\end{defn}

Thus, we have seen how to regard the compression of an operator system by a single projection as a quotient of the original operator system. We now wish to explore the compression of the operator system $M_2(\Cal V)$ relative to the cones $\{C(p_n \oplus q_n)\}_n$ where $p \in \Cal V$ is a projection with $q = I-p,$ and \begin{align}
    C(p_n \oplus q_n):= \{ x \in M_{2n}(\Cal V): x=x^*, (p_n \oplus q_n)x(p_n \oplus q_n) \in B(H^{2n})^+\}.
\end{align}
 We first prove the following quick lemma:
 
 \begin{lem}\label{lem: projections cones stable under direct sum}
Given an operator system $\Cal V \subset B(H)$ and projections $p,q \in \Cal V$ then $C(p_n) \oplus C(q_n) \subset C(p_n \oplus q_n)$ for all $n \in \bb N$.
\end{lem}

\begin{proof}
Given $x \in C(p_n)$ and $y \in C(q_n)$ then we see \begin{align*}
    (p_n\oplus q_n)(x \oplus y)(p_n \oplus q_n) = (p_nxp_n \oplus q_nyq_n)
\end{align*} and for any $(\eta_1,\eta_2) \in H^n \oplus H^n$ we see \begin{align*}
   \innerproduct{(p_nxp_n \oplus q_nyq_n)(\eta_1,\eta_2)}{(\eta_1,\eta_2)} = \innerproduct{p_nxp_n\eta_1}{\eta_1} + \innerproduct{q_nyq_n\eta_2}{\eta_2}  \in  \bb R^{+}
\end{align*} and thus $x\oplus y \in C(p_n \oplus q_n)$ which proves the result.
\end{proof}

We will use the ``Canonical shuffle'' (see Remark \ref{rem: the canonical shuffle}) as presented in the preliminary section. 

\begin{lem}\label{lem: the nth amplification of the compression (p,q) is (p_n,q_n)}
Given an operator system $\Cal V \subset B(H)$, with projection $p \in \Cal V, q = I -p$, consider the operator $\vp: M_2(\Cal V) \to B(H^2)$ defined by $\vp(x) = (p\oplus q)x(p \oplus q).$ Then $I_n \otimes (p \oplus q) \rightsquigarrow (p_n \oplus q_n)$ under the $*$-isomorphism $M_n(M_2(\Cal V)) \simeq M_2(M_n(\Cal V)).$ This is to say that the nth-amplification of $\vp$ is given by $p_n \oplus q_n$ under the canonical shuffle. 
\end{lem}
\begin{proof}
Fix $a \in M_n(M_2(\Cal V)).$ We then write $a = \sum_{ij} \ket{i}\bra{j} \otimes a_{ij}, a_{ij} \in M_2(\Cal V)$ and see \begin{align*}
    \vp_n(\sum_{ij} \ket{i}\bra{j} \otimes a_{ij}) &= \sum_{ij} \ket{i}\bra{j} \otimes \vp(a_{ij}) = \sum_{ij} \ket{i}\bra{j} \otimes (p \oplus q)(a_{ij})(p\oplus q) \\
    &= \sum_{ij} \ket{i}\bra{j} \otimes [\ket{1}\bra{1} \otimes pa_{ij11}p + \ket{1}\bra{2} \otimes pa_{ij12}q + \ket{2}\bra{1} \otimes qa_{ij21}p + \ket{2}\bra{2} \otimes qa_{ij22}q],
\end{align*} where $\{\ket{k}\bra{l}\}_{k,l}$ in the bracket above denote the matrix units in $M_2.$ It then follows \begin{align*}
    &\sum_{ij} \ket{i}\bra{j} \otimes [\ket{1}\bra{1} \otimes pa_{ij11}p + \ket{1}\bra{2} \otimes pa_{ij12}q + \ket{2}\bra{1} \otimes qa_{ij21}p + \ket{2}\bra{2} \otimes qa_{ij22}q] \\
    =&\kb{1}{1} \otimes \sum_{ij} \kb{i}{j} \otimes pa_{ij11}p + \kb{1}{2} \otimes \sum_{ij} \kb{i}{j} \otimes pa_{ij12}q \\
    +&\kb{2}{1} \otimes \sum_{ij} \kb{i}{j} \otimes qa_{ij21}p + \kb{2}{2} \otimes \sum_{ij} \kb{i}{j} \otimes qa_{ij22}q.
\end{align*} Let $a_{kl} = \sum_{ij} \kb{i}{j} \otimes a_{ijkl}.$ We finally get \begin{align*}
    &\kb{1}{1} \otimes \sum_{ij} \kb{i}{j} \otimes pa_{ij11}p + \kb{1}{2} \otimes \sum_{ij} \kb{i}{j} \otimes pa_{ij12}q \\ 
    +&\kb{2}{1} \otimes \sum_{ij} \kb{i}{j} \otimes qa_{ij21}p + \kb{2}{2} \otimes \sum_{ij} \kb{i}{j} \otimes qa_{ij22}q \\
    &=(p_n\oplus q_n)(\sum_{kl} \kb{k}{l} \otimes a_{kl})(p_n\oplus q_n)
\end{align*} which proves the result.
\end{proof}

\begin{lem}\label{lem: compatibility of C(p_n oplus q_n) on V oplus V}
Given an operator system $\Cal V \subset B(H)$ with $p \in \Cal V$ a projection. Then $\{C(p_n \oplus q_n)\}_n$ is a matrix ordering on $\Cal V \oplus \Cal V$.
\end{lem}

\begin{proof}
We begin with an observation. By Lemma \ref{lem: projections cones stable under direct sum} we know $C(p_n) \oplus C(q_n) \subset C(p_n \oplus q_n).$ If we consider $x \oplus y \in M_n(\Cal V) \oplus M_n(\Cal V)$ then if $x\oplus y \in C(p_n \oplus q_n)$ we see \begin{align*}
    (p_n \oplus q_n)(x\oplus y)(p_n \oplus q_n) = \begin{pmatrix}
    p_nxp_n & 0 \\
    0 & q_nyq_n
    \end{pmatrix} \in B(H^{2n})^+.
\end{align*} Compression by $\begin{pmatrix}
1 \\
0
\end{pmatrix}$ implies that $p_nxp_n \in B(H^n)^+$ and therefore $x \in C(p_n).$ Similarly, compression by $\begin{pmatrix}
0 \\
1
\end{pmatrix}$ implies that $q_nyq_n \in B(H^n)^+.$ Thus $x\oplus y \in C(p_n) \oplus C(q_n).$ Thus, we have $C(p_n) \oplus C(q_n) = C(p_n \oplus q_n)$ when restricting to elements of the form $x \oplus y.$

$*$-closed follows by definition. Fix $n \in \bb N$. It is immediate that given any $\lambda >0$ that $\lambda C(p_n \oplus q_n) \subset C(p_n \oplus q_n)$. Let $a,b \in C(p_n \oplus q_n)$. Since the compression by the projection $p_n\oplus q_n$ is linear, we have that $(p_n\oplus q_n)(a+b)(p_n\oplus q_n)  = (p_n\oplus q_n)a(p_n\oplus q_n)+ (p_n\oplus q_n)b(p_n\oplus q_n) \in B(H^{2n})^+ + B(H^{2n})^+ \subset B(H^{2n})^+.$ We now check compatibilty. Let $m \in \bb N$ with $\alpha \in M_{2n,2m}$ and consider $\alpha^*(x\oplus y)\alpha \in M_2(M_m(\Cal V))$, where $x\oplus y \in C(p_n \oplus q_n)$. Let $\vp: M_2(\Cal V) \to B(H^2)$ be defined as in Lemma \ref{lem: the nth amplification of the compression (p,q) is (p_n,q_n)}. Then \begin{align*}
    &(p_m \oplus q_m)\alpha^*(x \oplus y)\alpha(p_m \oplus q_m) = \vp_m(\alpha^*(x \oplus y)\alpha) = \alpha^*\vp_n(x \oplus y) \alpha \\
    &= \alpha^*(p_n \oplus q_n)(x \oplus y)(p_n \oplus q_n)\alpha \in \alpha^* B(H^{2n})^+\alpha \subset B(H^{2m})^+.
\end{align*} which proves the result.

\end{proof}

We consider the family $\{\wt{C}(p_n \oplus q_n)\}_n$ where for each $n \in \bb N$,\begin{align*}
    \wt{C}(p_n\oplus q_n) = \{ ((x_{ij} \oplus y_{ij})+ J_{p \oplus q}) \in M_{n}((\Cal V \oplus \Cal V)/ J_{p \oplus q}): (x\oplus y) \in C(p_n \oplus q_n), x = (x_{ij}), y = (y_{ij})\}.
\end{align*} Here we have let $\Cal V \oplus \Cal V$ denote the $*$-vector subspace of $M_2(\Cal V)$ consisting of all diagonal matrices over $\Cal V.$

\begin{prop}
Given an operator system $\Cal V \subset B(H)$ and a projection $p \in \Cal V$. Let $q = I_H - q.$ Then $(\Cal V \oplus \Cal V) / J_{p \oplus q}$ is an operator system.
\end{prop}

\begin{proof}
By Lemma \ref{lem: compatibility of C(p_n oplus q_n) on V oplus V} it is immediate to show that $\{\wt{C}(p_n \oplus q_n)\}_n$ is a matrix ordering. Another direct consequence of the definition is that given any $((x_{ij} \oplus y_{ij}) + J_{p \oplus q})_{ij} \in C(p_n \oplus q_n) \cap -C(p_n \oplus q_n)$ then $x_{ij} \in C(p) \cap -C(p)$ and $y_{ij} \in C(q) \cap -C(q)$ for all $i,j$ which, along with Lemma \ref{lem: projections cones stable under direct sum}, proves the cones are proper. The claim that $p \oplus q + J_{p \oplus q}$ is an Archimedean order unit follows from the results that $p$ is an Arhimedean order unit for $\{C(p_n)\}_n$, $q$ is an Archimedean order unit for $\{C(q_n)\}_n$ and then applying Lemma \ref{lem: projections cones stable under direct sum} and the observation in Lemma \ref{lem: compatibility of C(p_n oplus q_n) on V oplus V}. We leave the details to the reader.
\end{proof}

More generally it will follow that $(p \oplus q)M_2(\Cal V)(p \oplus q)$ is an operator system. We begin with a lemma:

\begin{lem}\label{lem: compatibility C(p_n oplus q_n) on M_2(V)}
Given an operator system $\Cal V \subset B(H)$ with projections $p,q \in \Cal V$ then $\{C({p_n \oplus q_n})\}_n$ is a matrix ordering on $M_2(\Cal V).$
\end{lem}

\begin{proof}
This proof is the same as Lemma \ref{lem: compatibility of C(p_n oplus q_n) on V oplus V} but we will make some remarks for completeness. Once again, $*$-closed, and positive homogeniety of the family is immediate. Let $\vp: M_2(\Cal V) \to B(H^2)$ once again denote the linear map given by $x \mapsto (p\oplus q)x(p\oplus q).$ Then compatibility is immediate from Lemma \ref{lem: the nth amplification of the compression (p,q) is (p_n,q_n)}. This finishes the proof.
\end{proof}

\begin{thm}\label{thm: the compression operator system M_2(V)}
Given an operator system $\Cal V \subset B(H)$ and a projection $p \in \Cal V,$ with $q = I-p$ then \[ (M_2(\Cal V) / J_{p_\oplus q}), \{\wt{C}(p_n \oplus q_n)\}_n, (p\oplus q) + J_{p \oplus q}) \] is an operator system.
\end{thm}

\begin{proof}
The fact that $J_{p \oplus q}$ is a subspace along with Lemma \ref{lem: compatibility C(p_n oplus q_n) on M_2(V)} immediately gives that the family $\{\wt{C}(p_n \oplus q_n)\}_n$ is a matrix ordering. Given any $(x_{ij})_{ij} \in M_n(M_2(\Cal V))$ one can show that for each $i,j$ that \begin{align}
     x_{ij} = \begin{pmatrix}
    x_{ij11} & x_{ij12} \\
    x_{ij21} & x_{ij22}
    \end{pmatrix}  = \begin{pmatrix}
    x_{ij11} & 0 \\
    0 & x_{ij22}
    \end{pmatrix} + \begin{pmatrix}
    0 & x_{ij12} \\
    x_{ij21} & 0
    \end{pmatrix} \in J_{p \oplus q} + J_{p \oplus q} \subset J_{p \oplus q}.
\end{align} Thus $\wt{C}(p_n \oplus q_n) \cap - \wt{C}(p_n \oplus q_n)$ has trivial intersection which implies the family $\{\wt{C}(p_n \oplus q_n)\}_n$ is proper. An exercise in matrix mechanics shows that $p \oplus q + J_{p \oplus q}$ is an Archimedean matrix order unit. This finishes the proof.
\end{proof}

\begin{cor}\label{cor: concrete compression by finite family of projections}
Consider a family of projections $\{p^i\}_{i=1}^N \subset \Cal V,$ and let $q^i = I - p^i$ for all $i$. Let $P := \oplus_i p^i, Q:= \oplus_i  q^i$. Then \begin{align*}
    M_{2N}(\Cal V) / J_{P \oplus Q}
\end{align*} is an operator system.
\end{cor}

\begin{proof}
This is immediate since $P, Q \in M_N(\Cal V)$ are projections (as operators in $B(H^N)$) with of course $I_N - P = \oplus_i I - p^i = \oplus_i q^i  = Q,$ and after making the identification $M_{2N}(\Cal V) = M_2(M_N(\Cal V)).$ Thus we see \begin{align*}
    (P\oplus Q)M_{2N}(\Cal V)(P \oplus Q) = (P\oplus Q)M_{2}(M_N(\Cal V))(P \oplus Q)
\end{align*}
is an operator system by Theorem \ref{thm: the compression operator system M_2(V)}.
\end{proof}

\section{Abstract compression operator systems} \label{sec: abstract compression operator systems}
Motivated by Section \ref{sec: concrete compression operator systems} we now wish to consider compression operator systems in an abstract sense.
We begin by considering a $*$-vector space with structure similar to that of an operator system, but lacking proper cones. We will show that a natural quotient of such a $*$-vector space is in fact an operator system. We will then consider a particular case of a $*$-vector space with such a matrix ordering induced by positive contractions. Finally we will show that this structure coincides with the structure of an operator system compressed by a projection, as studied in Section \ref{sec: concrete compression operator systems}.

Let $\Cal V$ be a $*$-vector space, and let $\{C_n\}_n$ be a matrix ordering on $\Cal V$. For every $n \in \mathbb{N}$, let $J_n := \Span C_n \cap -C_n$. 

\begin{lem} Suppose that $x \in J_n$. Then $x = a + ib$ for some $a,b \in C_n \cap -C_n$. \end{lem}

\begin{proof}
If $x \in J_n$, then $x = \sum_k \lambda_k c_k$ where each $c_k \in C_n \cap -C_n$ and each $\lambda_k \in \mathbb{C}$. Assume $\lambda_k = r_k + is_k$ for $r_k,s_k \in \mathbb{R}$. Then $x = (\sum_k r_k c_k) + i(\sum_k s_k c_k)$. So $a = (\sum_k r_k c_k) \in C_n \cap -C_n$ and $b = (\sum_k s_k c_k) \in C_n \cap -C_n$.
\end{proof}

For convenience, set $J := J_1$. Compare the next result with Lemma \ref{lem: M_n(J_p) = J_(p_n) }.

\begin{lem} \label{lem: M_n(J)=J_n abstract setting}
For every $n \in \mathbb{N}$, $M_n(J) = J_n$. 
\end{lem}

\begin{proof}
We first show that $M_n(J) \subseteq J_n$. It is necessary and sufficient to show that for every $j,k \leq n$ and $x \in J$ we have $\ket{j} \bra{k} \otimes x \in J_n$. Let $x \in J$. By the previous lemma, $x = a + ib$ for $a,b \in C_1 \cap -C_1$. Then $\ket{j} \bra{j} \otimes a, \ket{j} \bra{j} \otimes b \in C_n \cap -C_n$ since $\{C_n\}$ is a matrix cone. Hence $\pm \ket{j} \bra{j} \otimes x \in J_n$. Similarly, $(\ket{j} + \ket{k})(\bra{j} + \bra{k}) \otimes x \in J_n$ and $(\ket{j} + i \ket{k})(\bra{j} -i\bra{k})\otimes x \in J_n$. So \[ ((\ket{j} + \ket{k})(\bra{j} + \bra{k}) - \ket{j}\bra{j} - \ket{k}\bra{k})\otimes x = (\ket{j}\bra{k} + \ket{k}\bra{j}) \otimes x \in J_n  \] and \[ i((\ket{j} + i \ket{k})(\bra{j} -i\bra{k}) - \ket{j}\bra{j} - \ket{k}\bra{k}) \otimes x = (\ket{j}\bra{k} - \ket{k}\bra{j}) \otimes x \in J_n. \] It follows that $\ket{j} \bra{k} \otimes x \in J$.

Next we show $J_n \subseteq M_n(J)$. Let $x = (x_{jk}) \in J_n$. Then it is necessary and sufficient to show that $x_{jk} \in J$ for every $j,k$. By the previous lemma, $x = a + ib$ for $a,b \in J_n$, so it suffices to show that $a_{jk}$ and $b_{jk}$ are elements of $J$. We will show that $a_{jk} \in J$, and the proof for $b_{jk}$ is identical. Since $a \in C_n \cap -C_n$ and $\{C_n\}_n$ is a matrix cone, we have $a_{jj} = \bra{j} a \ket{j}, a_{kk} = \bra{k} a \ket{k} \in J$. Also $(\bra{j} + \bra{k})a(\ket{j} + \ket{k}) = a_{jj} + a_{jk} + a_{kj} + a_{kk} \in J$, and $(\bra{j} + i \bra{k})a(\ket{j} - i \ket{k}) = a_{jj} - ia_{jk} + ia_{kj} + a_{kk} \in J$. It follows that $a_{jk} + a_{kj}$ and $a_{jk} - a_{kj}$ are elements of $J$ and hence $a_{jk} \in J$. Similarly $b_{jk} \in J$ and hence $x_{jk} \in J$.
\end{proof}

Lemma \ref{lem: M_n(J)=J_n abstract setting} allows us to identify the vector spaces $M_n(\Cal V/J)$, $M_n(\Cal V) / M_n(J)$ and $M_n(\Cal V) / J_n$. Define \[ \tilde{C}_n := \{(x_{ij} + J) \in M_n(\Cal V/ J): x=(x_{ij}) \in C_n\}. \]

\begin{lem} \label{lem: abstract tilde C_n are matrix ordering}
The sequence $\{\tilde{C}_n\}$ is a matrix ordering on $\Cal V/J$. 
\end{lem}

\begin{proof}
That $\lambda \wt{C}_n \subset \wt{C}_n$ and $\wt{C}_n + \wt{C}_n$ are immediate. To check compatibility, fix $n \in \bb N$ and let $\alpha \in M_{n,m}$ and $(x_{ij} + J) \in \wt{C}_n$. Then \begin{align*}
   \alpha^*(x_{ij}+J)\alpha &= \sum_{i,j}\kb{i}{j} \otimes \sum_{k,l}\oline{\alpha}_{ki}(x_{kl}+J)\alpha_{lj} \\
   &= \sum_{i,j}\kb{i}{j} \otimes (\sum_{k,l}\oline{\alpha}_{ki}x_{kl}\alpha_{lj}+J) \\
   &= (\sum_{k,l}\oline{\alpha}_{ki}x_{kl}\alpha_{lj} + J)_{ij},
   \end{align*}and $\alpha^*x\alpha \in \alpha^*C_n\alpha \subset C_m.$ This finishes the proof.
\end{proof}

Suppose $\{C_n\}_n$ is a (not necessarily proper) matrix ordering on a $*$-vector space $\Cal V$. Recall that $e \in \Cal V$ a \it{matrix order unit} for $(\Cal V,\{C_n\}_n)$ if  for every $x \in M_n(\Cal V)$ with $x^*=x$, there exists $r > 0$ such that $x + re_n \in C_n$. If $x + \epsilon e_n \in C_n$ for all $\epsilon > 0$ implies that $x \in C_n$, we call $e$ \textit{Archimedean}.

\begin{prop} \label{prop: abstract quotent system}
Suppose that $\Cal V$ is a $*$-vector space with matrix ordering $\{C_n\}_n$ and an Archimedean matrix order unit $e$. Then $(\Cal V/J, \{\tilde{C}_n\}_n, e + J)$ is an operator system.
\end{prop}

\begin{proof}
By Lemma \ref{lem: abstract tilde C_n are matrix ordering} we need only show that the family $\{\wt{C}_n\}_n$ is proper and that $e + J$ is an Archimedean matrix order unit. If $(x_{ij}+ J) \in \wt{C}_n \cap -\wt{C}_n$ then $\pm x \in C_n$ which implies $x \in J_n = M_n(J)$ by Lemma \ref{lem: M_n(J)=J_n abstract setting}. Thus, $x_{ij} \in J$ for all $i,j$ and therefore $(x_{ij} +J) = (0 + J).$

It remains to show that $e+J$ is an Archimedean  matrix order unit. Let $(x_{ij} + J) \in M_n(\Cal V/ J)$ be $*$-hermitian and choose $r>0$ such that $re_n -x \in C_n$. Then \begin{align*}
    r(I_n \otimes (e+ J)) - (x_{ij}+J) = (I_n \otimes (re + J)) - (x_{ij}+J)= (r(e_n)_{ij} + J)- (x_{ij}+J) = (r(e_n)_{ij} - x_{ij} + J) \in \wt{C}_n
\end{align*} since $re_n -x \in C_n.$
 
 Finally, if $\epsilon I_n\otimes(e+J) + (x_{ij} + J) \in \wt{C}_n$ for all $\epsilon>0$ then $(\epsilon(e_n)_{ij}+x_{ij} + J) \in \wt{C}_n$ for all $\epsilon>0$ and by definition $\epsilon e_n + x \in C_n$ giving us $x \in C_n$ and thus $(x_{ij}+J) \in \wt{C}_n$. This finishes the proof.
\end{proof}

We now briefly return to the structural properties induced by projections in a concrete operator system $\Cal V \subset B(H).$ The following lemma will motivate Definition \ref{defn: positive cones relative to projection p} below for a matrix ordering $\{C(p_n)\}$ when $p$ is only a positive contraction in an abstract operator system.

\begin{lem}\label{lem: abstract vs concrete cones induced by p} Let $\Cal V \subset B(H)$ be an operator system and suppose that $p \in \Cal V$ is a projection. Then for any $x \in \Cal V$ with $x = x^*$, we have that $pxp \geq 0$ in $B(H)$ if and only if for every $\epsilon > 0$ there exists a $t > 0$ such that \[ x + \epsilon p + t (I - p) \geq 0. \] \end{lem} 

\begin{proof} First assume that for every $\epsilon > 0$ there exists a $t > 0$ such that $y = x + \epsilon p + t (I - p) \geq 0$. Then the compression of $y$ by $p$ is positive. Hence $pyp = pxp + \epsilon p \geq 0$. Since this holds for all $\epsilon > 0$ and since $p$ is an Archimedean order unit for the set of operators of the form $pzp$ (see Proposition \ref{prop: Properties of C(p)}) it follows that $pxp \geq 0$.

Now assume that $pxp \geq 0$, and let $\epsilon > 0$. Let $ q = I - p$. It follows that if we write $H = pH \oplus qH$, we see that (by exercise 3.2(i) in \cite{paulsen2002completely}) an operator $T$ is positive if and only if $pTp \geq0$, $qTq \geq 0$, and for every $h \in pH$ and $k \in qH$ we have that \[ \abs{\innerproduct{pTqk}{h}}^2 \leq \innerproduct{pTp h}{h}\innerproduct{qTq k}{k}. \] 
Now choose $t > \|x\|$ such that $\epsilon(t - \|x\|) > \|x\|^2$ and consider $T = x + \epsilon p + tq$. Then $qTp = qxp$, $pTp = pxp + \epsilon p \geq 0$, and $qTq = tq + qxq \geq tq - \|x\|q \geq 0$. Moreover \[ |\innerproduct{pTq k}{h}|^2 \leq \|x\|^2 \|h\|^2 \|k\|^2 \leq \epsilon(t - \|x\|) \|h\|^2 \|k\|^2 \leq \innerproduct{pTp h}{h}\innerproduct{qTq k}{k}, \] since $\epsilon \|h\|^2 = \innerproduct{\epsilon p h}{h} \leq \innerproduct{pTp h}{h}$ and $(t-\|x\|)\|k\|^2 = \innerproduct{(tq - \|x\|q) k}{k} \leq \innerproduct{qTq k}{k}$. So $x + \epsilon p + tq \geq 0$. \end{proof}

This lemma thus relates positivity of the compression by $p$ to positivity in the operator system $\Cal V.$ This motivates us to make the following definition. 

\begin{defn}\label{defn: positive cones relative to projection p} Let $(\Cal V,\{C_n\}_n, e)$ be an operator system, and suppose that $p \in \Cal V$ with $0 \leq p \leq e$, i.e., let $p \in \Cal V$ be a positive contraction of $\Cal V$. For each $n \in \bb N$ and let $p_n= I_n \otimes p$. We define the \
\textit{positive cone relative to $p_n$}, denoted $C(p_n)$, to be \begin{align}
    C(p_n) :&= \{ x \in M_n(\Cal V) : x = x^*, \text{ for all } \epsilon > 0 \text{ there exists } t > 0 \text{ such that } x + \epsilon p_n + t(e_n - p_n) \in C_n \}.
    \end{align} 
    \end{defn}
    
An immediate consequence of Lemma \ref{lem: abstract vs concrete cones induced by p} is that if a positive contraction $p \in \Cal V$ is a projection, then for each $n \in \bb N$ the positive cone relative to $p_n$ becomes 
    \begin{align}
        C(p_n) = \{ x \in M_n(\Cal V): x=x^*, p_nxp_n \in B(H^{n})^+\},
    \end{align}

We now prove a similar string of results mirroring those of Section \ref{sec: concrete compression operator systems}.
\begin{prop} \label{prop: properties of C(P) for any positive contraction P} Let $\Cal V$ be an operator system and $0 \leq p \leq e$. Then the sequence $\{C(p_n)\}_n$ is a matrix ordering for $\Cal V$, and $p$ is an Archimedean matrix order unit for $\{C(p_n)\}_n$.
\end{prop}

\begin{proof}
We first check that $\{C(p_n)\}_n$ is a matrix ordering. As noted in the preliminaries, it suffices to check that For each $n,m \in \mathbb{N}$, $C(p_n) \oplus C(p_m) \subseteq C(p_{n+m})$, and for each $n,m \in \mathbb{N}$ and $\alpha \in M_{n,m}$, $\alpha^* C(p_n) \alpha \subseteq C(p_m)$. Suppose $x \in C(p_n)$ and $y \in C(p_m)$. Let $\epsilon > 0$. Then there exist $r_1, r_2 > 0$ such that $x + \epsilon p_n + r_1(e_n-p_n) \geq 0$ and $y + \epsilon p_m + r_2(e_m - p_m) \geq 0$. It follows that $(x\oplus y) + \epsilon p_{n+m} + \max(r_1,r_2) (e_{n+m} - p_{n+m}) \geq 0$. Now, let $\epsilon > 0$ and suppose that $x \in C(p_n)$. Also assume $\alpha \neq 0$. Then there exists $r > 0$ such that \[ x + \frac{\epsilon}{\|\alpha\|^2} p_n + r(e_n - p_n) \geq 0. \] It follows that \[ \alpha^* x \alpha + \frac{\epsilon}{\|\alpha\|^2} \alpha^* p_n \alpha + r \alpha^*(e_n - p_n) \alpha \geq 0. \] However, since $\alpha^* p_n \alpha \leq \|\alpha\|^2 p_m$ and $\alpha^*(e_n-p_n)\alpha \leq \|\alpha\|^2(e_m-p_m)$ we have \[ \alpha^* x \alpha + \epsilon p_m + (r\|\alpha\|^2)(e_m - p_m) \geq 0. \] It follows that $\alpha^* x \alpha \in C(p_m)$.

We now show that $p$ is an Archimedean matrix order unit for $\{C(p_n)\}_n$. We verify the relevant properties for the case for $n=1$ and for $n >1$ the proofs are similar. Choose $r > 0$ such that $x + re \geq 0$. Let $\epsilon > 0$. Then \[ (x + rp) + \epsilon p + r(e - p) = x + re + \epsilon p \geq 0. \] It follows that $x + rp \in C(p)$. Finally, assume $x + \delta p \in C(p)$ for all $\delta > 0$ and let $\epsilon > 0$. Then there exists $r > 0$ such that \[ (x + \epsilon/2 p) + \epsilon /2 p + r(e-p) \geq 0. \] It follows that $x + \epsilon p + r(e-p) \geq 0$. So $x \in C(p)$. \end{proof}

We now come to the main results of this section. Similarly to our notation in the Section \ref{sec: concrete compression operator systems}, given an operator system $\Cal V$, and a positive contraction $p \in \Cal V$ 
we consider the matrix ordering $\{C(p_n)\}_n$ and we let $J_p = \Span C(p) \cap -C(p).$ 

We recall the following definition. Given an Archimedean order unit space $\Cal V$ then the \it{minimal order norm} $\alpha_m$ on $\Cal V$ is defined for $x \in \Cal V$ by \begin{align}
    \alpha_m(x) = \sup \{ \abs{\vp(x)}: \vp \in \Cal S(\Cal V)\}
\end{align} where $\Cal S(\Cal V)$ denotes the set of states on $\Cal V$. It is not difficult to show that if $\alpha_o: \Cal V_h \to [0,\infty)$ denotes the order norm induced by $e$ given by \begin{align}
    \alpha_o(x) = \inf\{ t>0: te \pm x \in \Cal V^+\},
\end{align} then $\alpha_o = \alpha_m$ when restricted to $\Cal V_h.$ We refer the interested reader to \cite[Section 4]{paulsen2009vector} for the details.

\begin{prop} \label{prop: nontrivial order unit}
Let $\Cal V$ an operator system and let $p \in \Cal V$ be a nonzero positive contraction. Let $\alpha_m: \Cal V \to [0,\infty)$ denote the minimal order norm induced by $e$. Then $\alpha_m(p) = 1$ if and only if $p \notin J_p.$
\end{prop}

\begin{proof}
By the assumption that $p$ is a positive contraction we know that $\alpha_o(p) = \inf\{t>0: te - p \in \Cal V^+\} \leq 1.$ The assumption that $\alpha_m(p) = 1$ implies $\alpha_o(p) = 1$. We first show that $p \notin -C(p)$. Suppose the contrary. Then by definition for all $\epsilon >0$ there exists $t >0$ such that $-p + \epsilon p + t(e-p) \in \Cal V^+.$ In other words it must follow \begin{align*}
    p \leq \frac{t}{1 + t - \epsilon }e.
\end{align*} If $\epsilon < 1$ then for all $ t >0$ we have $\frac{t}{1+t-\epsilon} < 1$ which contradicts the assumption that $\alpha_o(p) = 1.$ Thus $p \notin -C(p).$ Now suppose that $p \in J_p$. Since $p^* = p$ we have $p \in C(p) \cap -C(p)$, a contradiction. 
\end{proof}

As in Proposition \ref{prop: abstract quotent system} we will define the family of sets $\{\wt{C}(p_n)\}_n$ where for each $n \in \bb N$ we have \begin{align}
    \wt{C}(p_n)=\{ (x_{ij} + J_p) \in M_n(\Cal V/ J_p): x=(x_{ij}) \in C(p_n)\}.
\end{align} We now have the abstract analogue to Theorem \ref{thm: the compression of an operator system is an operator system}.

\begin{thm}\label{thm: abstract compression}
Given an operator system $\Cal V$ and positive contraction $p \in \Cal V$ such that $\alpha_m(p) = 1$, the triple $$\opsys{\Cal V/J_p}{\wt{C}(p_n)}{p + J_p}$$ is a non-trivial operator system. 
\end{thm}

\begin{proof}
By Proposition \ref{prop: properties of C(P) for any positive contraction P} we already have that $\{C(p_n)\}_n$ is a matrix ordering and $p$ is an Archimedean matrix order unit for the pair $(\Cal V, \{C(p_n)\}_n).$ By Proposition \ref{prop: abstract quotent system}, we deduce that $(\Cal V / J_p, \{\tilde{C}(p_n)\}_n, p + J_p)$ is an operator system. Moreover, this operator system is non-trivial by Proposition \ref{prop: nontrivial order unit}, since $\alpha_m(p)=1$.
\end{proof}

Combining Theorem \ref{thm: abstract compression}, Lemma \ref{lem: abstract vs concrete cones induced by p} and Theorem \ref{thm: the compression of an operator system is an operator system} yields the following key observation.

\begin{cor} \label{cor: concrete and abstract compression systems agree}
Suppose that $\Cal V \subset B(H)$ is an operator system and that $p \in \Cal V$ is a projection in $B(H)$. Then the abstract compression $(\Cal V/J_p, \{\tilde{C}(p_n)\}_{n \in \mathbb{N}}, p + J_p)$ is completely order isomorphic to the concrete compression $p \Cal V p$.
\end{cor}

Corollary \ref{cor: concrete and abstract compression systems agree} justifies the terminology we will use in the following definition.

\begin{defn}
Given an operator system $\Cal V$ and a positive contraction $p \in \Cal V$ such that $\alpha_m(p) = 1$ then we call the operator system $\opsys{\Cal V/J_p}{\wt{C}(p_n)}{p + J_p}$ the \it{abstract compression operator system} and denote it by $\Cal V/ J_p.$ 
\end{defn}

In the next section of the paper we will make use of the structure of the abstract compression operator system $M_2(\Cal V) / J_{p \oplus q}$ where $p$ is a positive contraction and $q = e - p$. We denote the positive cones relative to the positive contraction $p \oplus q$ by $\{C((p \oplus q)_n)\}_n$ where for each $n \in \bb N$ \begin{align}
    C((p \oplus q)_n) = \{ x \in M_{2n}(\Cal V): x=x^*, \forall \epsilon >0 \,\, \exists t >0 \,\,\tx{such that}\,\, x + \epsilon((p\oplus q)_n) + t((q\oplus p)_n) \in C_{2n}\},
\end{align} and $J_{p \oplus q} = \Span C(p \oplus q) \cap -C(p \oplus q).$ The concrete analogue of the following corollary was stated in Corollary \ref{cor: concrete compression by finite family of projections}.

\begin{cor}\label{cor: abstract compression by finite family of projections}
Given an operator system $\Cal V$ and a finite family of positive contractions $\{p^i\}_{i=1}^N \subset \Cal V^+$ such that $\alpha_m(p^i) = 1$ for some $i$, with $q^i = e - p^i$ for all $i,$, let $P = \oplus_i p^i,$ and $Q = \oplus q^i.$ Then $M_{2N}(\Cal V)/ J_{P \oplus Q}$ is an operator system.
\end{cor}

\section{Projections in operator systems}\label{sec: projections in operator systems}

In this section we will develop an abstract characterization for projections in operator systems. We start with the following useful observation.

\begin{lem}\label{lem: second characterization of C(p) v2} Let $\Cal V \subset B(H)$ be an operator system and suppose that $p \in \Cal V$ is a projection (when viewed as an operator on $H$). If $q = I - p$ and $a,b,c \in \Cal V$ with $a^* = a$ and $b^*=b$, then $pap + pcq + qc^*p + qbq \in B(H)^+$ if and only if \[ \begin{pmatrix} a & c \\ c^* & b \end{pmatrix} \in C(p \oplus q). \]
\end{lem}

\begin{proof}
First suppose that \[ \begin{pmatrix} a & c \\ c^* & b \end{pmatrix} \in C(p \oplus q). \] Then by Lemma \ref{lem: abstract vs concrete cones induced by p} we know that \[ \begin{pmatrix} p & 0 \\ 0 & q \end{pmatrix} \begin{pmatrix} a & c \\ c^* & b \end{pmatrix} \begin{pmatrix} p & 0 \\ 0 & q \end{pmatrix} \geq 0. \] Conjugating this matrix by the scalar matrix $\begin{pmatrix} 1 \\ 1 \end{pmatrix}$ yields the expression $pap + pcq + qc^*p + qbq$ and hence $pap + pcq + qc^*p + qbq \geq 0$.

Now suppose that $pap + pcq + qc^*p + qbq \geq 0$. Again, by Lemma \ref{lem: abstract vs concrete cones induced by p} it suffices to prove that the operator \[ T = \begin{pmatrix} p & 0 \\ 0 & q \end{pmatrix} \begin{pmatrix} a & c \\ c^* & b \end{pmatrix} \begin{pmatrix} p & 0 \\ 0 & q \end{pmatrix} \] is positive. To this end, let $h,k \in H$. Define $h_1 = ph$ and $h_2 = qk$. Note that $\innerproduct{h_1}{h_2}= 0$ since $p$ and $q$ are orthogonal projections. Let $\tilde{h} = h_1 + h_2$. Then \begin{eqnarray} \innerproduct{T (h \oplus k)}{(h \oplus k)} & = & \innerproduct{\begin{pmatrix} a & c \\ c^* & b \end{pmatrix} \begin{pmatrix} h_1 \\ h_2 \end{pmatrix}}{\begin{pmatrix} h_1 \\ h_2 \end{pmatrix}} \nonumber \\ & = & \innerproduct{(pap + pcq + qc^*p + qbq) \tilde{h}}{\tilde{h}} \geq 0. \nonumber \end{eqnarray} We conclude that \[ \begin{pmatrix} a & c \\ c^* & b \end{pmatrix} \in C(p \oplus q). \]
\end{proof}

In Section \ref{sec: concrete compression operator systems}, we observed that $\{C(p_n \oplus q_n)\}$ was a matrix ordering on $M_2(\Cal V)$. This is due to Lemma \ref{lem: the nth amplification of the compression (p,q) is (p_n,q_n)} which shows that $C((p \oplus q)_n)$ can be identified with $C(p_n \oplus q_n)$ via the canonical shuffle map. The next Lemma is an abstract variation on the same result.

\begin{lem} \label{lem: canonical shuffle in the abstract setting}
Let $\phi:M_n(M_2(\Cal V)) \to M_2(M_n(\Cal V))$ denote the canonical shuffle map. Then $\phi(C((p \oplus q)_n) = C(p_n \oplus q_n)$ and hence $x \in C((p \oplus q)_n)$ if and only if $\phi(x) \in C(p_n \oplus q_n)$.
\end{lem}

\begin{proof}
As noted in Remark \ref{rem: the canonical shuffle}, the canonical shuffle can be written as $\phi = \psi \otimes id: M_n(M_2) \otimes \Cal V \to M_2(M_n) \otimes \Cal V$ where $\psi: M_n(M_2) \to M_2(M_n)$ is a $*$-isomorphism and $id: \Cal V \to \Cal V$ is the identity map. Hence it is a complete order embedding. The statement follows from the observation that \[\phi(x + \epsilon(p \oplus q)_n + t(q \oplus q)_n) = \phi(x) + \epsilon p_n \oplus q_n + t q_n \oplus p_n. \]
\end{proof}

 Our abstract characterization for projections is based on the following Theorem.

\begin{thm} \label{thm: concrete part of characterization}
Let $\Cal V \subset B(H)$ be an operator system and $p \in \Cal V$ be a projection. Set $q = I-p$. Then for every $n \in \mathbb{N}$ and $x \in M_n(\Cal V)$ we have that $x \in C_n$ if and only if \[ \begin{pmatrix} x & x \\ x & x \end{pmatrix} \in C(p_n \oplus q_n). \]
\end{thm}

\begin{proof}
Suppose that $x \in C_n$ which implies that $p_nxp_n, q_nxq_n \geq 0.$ Applying Lemma \ref{lem: second characterization of C(p) v2} we have $\begin{pmatrix}
x & x \\
x & x
\end{pmatrix} \in C(p_n \oplus q_n)$ if and only if $p_nxp_n + p_nxq_n + q_nxp_n +q_n xq_n \geq 0$. It follows \begin{align*}
p_nxp_n + p_nxq_n + q_nxp_n +q_n xq_n = p_nx(p_n + q_n) + q_nx(p_n+q_n) = (p_n +q_n)x(p_n+q_n) = x \geq 0.
\end{align*} This proves one direction. Conversely, suppose that $\begin{pmatrix}
x & x \\
x & x
\end{pmatrix} \in C(p_n \oplus q_n).$ Once again by Lemma \ref{lem: second characterization of C(p) v2} this implies \begin{align*}
0 \leq p_nxp_n + p_nxq_n + q_nxp_n + q_nxq_n = (p_n + q_n)x(p_n+ q_n) = x.
\end{align*}
\end{proof}

\begin{defn} \label{defn: abstract projection}
Let $(\Cal V, \{C_n\}_n, e)$ be an abstract operator system and suppose that $p \in \Cal V^+ \backslash \{0\}$ with $p \leq e$ and $\alpha_m(p) = 1$. Set $q = e - p$. We call $p$ an \textit{abstract projection} if the map $\pi_p: \Cal V \to M_2(\Cal V)/ J_{p \oplus q}$ defined by \[ \pi_p: x \mapsto \begin{pmatrix} x & x \\ x & x \end{pmatrix} + J_{p \oplus q} \] is a complete order isomorphism onto its range.
\end{defn} We call $0$ the \it{zero} projection and necessarily $\alpha_m(0) = 0.$

\begin{rem}
Let $\Cal V$ be an operator system. Then $e \in \Cal V$ is an abstract projection. First note that $q = e - e = 0$ and $\alpha_m(e) = 1$. Finally we consider $\pi_e: \Cal V \to M_2(\Cal V)/ J_{e \oplus 0}$. $\pi_e$ is completely positive and if $(\pi_e)_n(x) \in \wt{C}_{2n}(e \oplus 0)$ then $\begin{pmatrix}
x & x \\
x & x
\end{pmatrix} \in C_{2n}(e \oplus 0)$ and thus for all $\epsilon >0$ there exists $t >0$ such that $\begin{pmatrix}
x & x \\
x & x
\end{pmatrix} + \epsilon(e \oplus 0)_n + t(0 \oplus e)_n \in C_{2n}.$ Compression to the (1,1) corner implies $x + \epsilon e \in C_n$ for all $\epsilon >0$ and therefore $x \in C_n$. Thus, we have $\pi_e$ is a complete order embedding. Thus, $e$ is an abstract projection. 
\end{rem}

\begin{rem}
It necessarily follows that $p$ is an abstract projection in an operator system $\Cal V$ if and only if $q = e - p$ is an abstract projection. Indeed, the assertion that $\pi_p$ is a order embedding  is equivalent to the assertion that $x \in C_n$ if and only if $\begin{pmatrix} 1 & 1 \\ 1 & 1 \end{pmatrix} \otimes x \in C_{2n}((p \oplus q)_n)$. In other words, $x \in C_n$ if and only if for every $\epsilon > 0$ there exists $t > 0$ such that \[ \begin{pmatrix} 1 & 1 \\ 1 & 1 \end{pmatrix} \otimes x + \epsilon (p \oplus q) \otimes I_n + t (q \oplus p) \otimes I_n \in C_{2n}. \] Conjugation by the unitary matrix $\begin{pmatrix} 0 & 1 \\ 1 & 0 \end{pmatrix} \otimes I_n$ shows that for every $\epsilon > 0$ there exists $t > 0$ such that \[ \begin{pmatrix} 1 & 1 \\ 1 & 1 \end{pmatrix} \otimes x + \epsilon (q \oplus p)_n + t (p \oplus q)_n \in C_{2n}. \] Consequently $\begin{pmatrix} 1 & 1 \\ 1 & 1 \end{pmatrix} \otimes x \in C_{2n}((q \oplus p)_n)$ if and only if $x \in C_n$. The reverse implication is the same argument.
\end{rem}

\begin{rem}
The requirement that $\alpha_m(p) = 1$ is clearly necessary for $p$ to be a projection. However, even when $\alpha_m(p)<1$ it is possible for $\alpha_m(q) = 1$ so that $M_2(V)/J_{p \oplus q}$ is non-trivial. It is instructive to see what happens in this case. So suppose $\alpha_m(p) < 1$, or equivalently $\|p\|<1$. Then for every $\epsilon < 1$ and for sufficiently large $t > 0$ we have 
\begin{equation} \label{eqn: strict contraction matrix inequality} -\begin{pmatrix} p & p \\ p & p \end{pmatrix} + \epsilon \begin{pmatrix} p & 0 \\ 0 & q \end{pmatrix}  + t \begin{pmatrix} q & 0 \\ 0 & p 
\end{pmatrix} = \begin{pmatrix} te + (\epsilon - t - 1)p & -p \\ -p & \epsilon e + (t - \epsilon - 1)p \end{pmatrix} \geq 0. \end{equation}
To see that this is true, first observe that because $\|p\| < 1$, we can choose $t$ sufficiently large so that
\[ p \left( \frac{1 - \epsilon + t}{t} \right) \leq e \text{ and } \|p\|^2 \leq t \epsilon \left( 1 - \frac{\|p\|(1 - \epsilon + t)}{t} \right). \]
The bound on $\|p\|^2$ holds because
\[ \left( 1 - \frac{\|p\|(1 - \epsilon + t)}{t} \right) \to 1 - \|p\| \]
as $t \to \infty$, and thus
\[ t \epsilon \left( 1 - \frac{\|p\|(1 - \epsilon + t)}{t} \right) \to \infty \]
as $t \to \infty$. Now let $\phi:V \to B(H)$ be any positive unital map. Then for any $h,k \in H$ we have
\begin{eqnarray}
    |\langle \phi(p) h, k \rangle|^2 & \leq & \|p\|^2 \|h\|^2 \|k\|^2 \nonumber \\
    & \leq & t \epsilon \left( 1 - \frac{\|p\|(1 - \epsilon + t)}{t} \right) \|h\|^2 \|k\|^2 \nonumber \\
    & = & (t - \|p\|(1 - \epsilon + t)) \|h\|^2 ( \epsilon \|k\|^2) \nonumber \\
    & \leq & (t - \|p\|(1 - \epsilon + t)) \|h\|^2 ( \epsilon \|k\|^2 + \langle (t - \epsilon - 1)\phi(p) k|  k \rangle ) \nonumber \\
    & \leq & (t \|h\|^2 - t\langle \frac{1 - \epsilon + t}{t} \phi(p) h| h \rangle ) (\epsilon \|k\|^2 + \langle(t - \epsilon - 1)\phi(p) k| k \rangle ) \nonumber \\
    & = & \langle (t I + (\epsilon - t - 1)\phi(p) h|  h \rangle \langle (\epsilon I + (t - \epsilon - 1) \phi(p)) k| k \rangle \nonumber
\end{eqnarray} 
We conclude that the matrix in line (\ref{eqn: strict contraction matrix inequality}) is positive. It follows that $\begin{pmatrix} p & p \\ p & p \end{pmatrix} \in J_{p \oplus q}$ and thus $\pi_p$ is not a complete order embedding (in fact, it is not injective).
\end{rem}
\indent Thus, similar to the case for C*-algebras, we see that for an operator system $\Cal V$, $p$ is a an abstract projection if and only if $q :=e-p$ is an abstract projection, and every non-trivial  abstract projection in $\Cal V$ has norm 1.

Theorem \ref{thm: concrete part of characterization}, together with Corollary \ref{cor: concrete and abstract compression systems agree} and Lemma \ref{lem: canonical shuffle in the abstract setting} imply that the map $\pi_p$ is a complete order isomorphism onto its range whenever $p$ is a projection in a concrete operator system. In other words, every concrete projection is an abstract projection. It remains to show that every abstract projection is a concrete projection under some complete order embedding of its containing operator system. We proceed by first showing that matrix-valued ucp maps on $M_2(\Cal V)/ J_{p \oplus q}$ can be modified to build new matrix-valued ucp maps sending $p \oplus 0 + J_{p \oplus q}$ and $0 \oplus q + J_{p \oplus q}$ to projections.

\begin{prop} \label{prop: projection representation} Suppose that $\Cal V$ is an operator system and that $p$ is an abstract projection in $\Cal V$. Let $q = e - p$. Then for every ucp map $\phi: M_2(\Cal V)/ J_{p \oplus q} \rightarrow M_n$ there exists a $k \in \mathbb{N}$ and a ucp map $\psi: M_2(\Cal V)/ J_{p \oplus q} \rightarrow M_{k}$ such that $\psi(p \oplus 0 + J_{p \oplus q})$ and $\psi(0 \oplus q + J_{p \oplus q})$ are projections and satisfying the property that \[ \phi_{2n} \left( \begin{pmatrix} a & 0 & 0 & b \\ 0 & 0 & 0 & 0 \\ 0 & 0 & 0 & 0 \\ b^* & 0 & 0 & c \end{pmatrix} + M_{2n}(J_{p \oplus q}) \right) = \begin{pmatrix} \phi_n \left( \begin{pmatrix} a & 0 \\ 0 & 0 \end{pmatrix} + M_n(J_{p \oplus q}) \right) & \phi_n \left( \begin{pmatrix} 0 & b \\ 0 & 0 \end{pmatrix} + M_n(J_{p \oplus q}) \right) \\ \phi_n \left( \begin{pmatrix} 0 & 0 \\ b^* & 0 \end{pmatrix} + M_n(J_{p \oplus q}) \right) & \phi_n \left( \begin{pmatrix} 0 & 0 \\ 0 & c \end{pmatrix} + M_n(J_{p \oplus q}) \right) \end{pmatrix} \geq 0 \] if and only if $\psi_n \left( \begin{pmatrix} a & b \\ b^* & c \end{pmatrix} + M_n(J_{p \oplus q}) \right) \geq 0$ for all $a,b,c \in M_n(\Cal V)$. \end{prop}

\begin{proof} To simplify notation, we will let $\hat{x}$ denote the coset $x + M_n(J_{p \oplus q})$ for each $x \in M_n(\Cal V)$ throughout the proof. Suppose $\phi: M_2(\Cal V)/ J_{p \oplus q} \rightarrow M_n$ is ucp. Since $\phi(\widehat{p \oplus 0}) = I_n - \phi(\widehat{0 \oplus q})$, we have that $\phi(\widehat{p \oplus 0})$ commutes with $\phi(\widehat{0 \oplus q})$. Thus we may find a common orthonormal basis for $\mathbb{C}^n$ such that $\phi(\widehat{p \oplus 0})$ and $\phi(\widehat{0 \oplus q})$ are both diagonal. By reordering this basis, we may assume \[ \phi(\widehat{p \oplus 0}) = \tilde{P} = \begin{pmatrix} I_m & & & & \\ & x_{m+1} & & & \\ & & \ddots & &  \\ & & & x_{m'} & \\ & & & & 0_{n-m'} \end{pmatrix}, \quad \phi(\widehat{0 \oplus q}) = \tilde{Q} = \begin{pmatrix} 0_m & & & & \\ & y_{m+1} & & & \\ & & \ddots & &  \\ & & & y_{m'} & \\ & & & & I_{n-m'} \end{pmatrix} \] where $0 \leq m \leq m' \leq n$ and $x_i + y_i = 1$ for each $i=m+1,\dots, m'$ and $x_i, y_i \in (0,1)$. Define rectangular matrices \[ V = \begin{pmatrix} I_m & & & & 0 & \dots & 0 \\ & x_{m+1}^{-1/2} & & & & & \\ & & \ddots & & & & \\ & & & x_{m'}^{-1/2} & 0 & \dots & 0 \end{pmatrix} \in M_{m',n}, \quad W = \begin{pmatrix} 0 & \dots & 0 & y_{m+1}^{-1/2} & & &  \\ & & & & \ddots & \\ & & & & & y_{m'}^{-1/2} \\ 0 & \dots & 0 & & & & I_{n-m'} \end{pmatrix} \in M_{n-m,n}. \] Thus $V \tilde{P} V^* = I_{m'}$ and $W \tilde{Q} W^* = I_{n-m}$. We may now define $\psi: M_2(V) / J_{p \oplus q} \rightarrow M_{m' + n - m}$ via \[ \psi \begin{pmatrix} a & b \\ c & d \end{pmatrix} = \begin{pmatrix} V & 0 \\ 0 & W \end{pmatrix} \begin{pmatrix} \phi \widehat{\begin{pmatrix} a & 0 \\ 0 & 0 \end{pmatrix}} & \phi \widehat{\begin{pmatrix} 0 & b \\ 0 & 0 \end{pmatrix}} \\ \phi \widehat{\begin{pmatrix} 0 & 0 \\ c & 0 \end{pmatrix}} & \phi \widehat{\begin{pmatrix} 0 & 0 \\ 0 & d \end{pmatrix}} \end{pmatrix} \begin{pmatrix} V^* & 0 \\ 0 & W^* \end{pmatrix}. \] Then $\psi$ is ucp, $\psi(\widehat{p \oplus 0}) = I_{m'} \oplus 0_{n-m}$ and $\psi(\widehat{0 \oplus q}) = 0_{m'} \oplus I_{n-m}$.

It remains to check the final statement of the proposition. To show this, it suffices to show that the non-zero entries of $\phi(\widehat{a \oplus 0})$ lie in its upper left $m' \times m'$ corner, the non-zero entries of $\phi(\widehat{0 \oplus c})$ lie in its lower right $(n-m) \times (n-m)$ corner, and the non-zero entries of \[ \phi \widehat{\begin{pmatrix} 0 & b \\ 0 & 0 \end{pmatrix}} \] lie in its upper right $m' \times (n-m)$ corner. Indeed, when these statements hold, the map $\psi$ is simply the compression of the matrix \[ \begin{pmatrix} \phi  \widehat{\begin{pmatrix} a & 0 \\ 0 & 0 \end{pmatrix}} & \phi \widehat{\begin{pmatrix} 0 & b \\ 0 & 0 \end{pmatrix}} \\ \phi \widehat{\begin{pmatrix} 0 & 0 \\ b^* & 0 \end{pmatrix}} & \phi \widehat{\begin{pmatrix} 0 & 0 \\ 0 & c \end{pmatrix}} \end{pmatrix} \] to the $(m' + n - m) \times (m' + n - m)$ submatrix upon which it is supported, followed by conjugation by the invertible matrix \[ \begin{pmatrix} I_m & & & & & & & \\ & x_{m+1}^{-1/2} & & & & & & \\ & & \ddots & & & & & \\ & & & x_{m'}^{-1/2} & & & & \\ & & & & y_{m+1}^{-1/2} & & & \\ & & & & & \ddots & & \\ & & & & & & y_{m'}^{-1/2} & \\ & & & & & & & I_{n-m'} \end{pmatrix}. \]

We first consider the coset of the matrix with $b$ in its upper right corner and zeroes elsewhere, where $b \in \Cal V$. Since $\widehat{p \oplus q}$ is an order unit for $M_2(\Cal V) / J_{p \oplus q}$, we may assume (by rescaling $b$ if necessary) that  \[ \begin{pmatrix} p & b \\ b^* & q \end{pmatrix} \in C(p \oplus q). \] This implies that \[ \begin{pmatrix} p & 0 & 0 & b \\ 0 & 0 & 0 & 0 \\ 0 & 0 & 0 & 0 \\ b^* & 0 & 0 & q \end{pmatrix} \in C(p_{2} \oplus q_{2}). \] The complete positivity of $\phi$ implies that \[ \begin{pmatrix} \tilde{P} & \phi \widehat{\begin{pmatrix} 0 & b \\ 0 & 0 \end{pmatrix}} \\ \phi \widehat{\begin{pmatrix} 0 & 0 \\ b^* & 0 \end{pmatrix}} & \tilde{Q} \end{pmatrix} \geq 0. \] The claim follows.

Next we consider $\phi(a \oplus 0)$ for $a = a^*$. By again rescaling $a$ as necessary, we may assume that \[ \begin{pmatrix} p \pm a & 0 \\ 0 & q \end{pmatrix} \in C(p \oplus q). \] By the definition of $C(p \oplus q)$ and compressing to the upper left corner, this implies that for every $\epsilon > 0$ there exists a $t > 0$ such that $p \pm a + \epsilon p + tq \geq 0$. Using the definition of $C(p \oplus q)$ again, we see that this implies \[ \begin{pmatrix} p \pm a & 0 \\ 0 & 0 \end{pmatrix} \in C(p \oplus q). \] By the positivity of $\phi$, this means that \[ \tilde{P} \pm \phi \widehat{\begin{pmatrix} a & 0 \\ 0 & 0 \end{pmatrix}} \geq 0. \] It follows that the non-zero entries of $\phi(\widehat{a \oplus 0})$ lie in its upper left $m' \times m'$ corner as claimed. A similar proof shows that $\phi(\widehat{0 \oplus c})$ has its non-zero entries in its lower right $(n-m) \times (n-m)$ corner whenever $c = c^*$.
\end{proof}

We prove one final lemma before arriving at the main result of this section. This lemma ensures that the map $\pi_p$ in Definition \ref{defn: abstract projection} is unital.

\begin{lem}\label{lem: the induced map pi is unital}
For any positive contraction $p$ in an operator system $\Cal V$ we have \[ \begin{pmatrix} p & p \\ p & p \end{pmatrix} + J_{p \oplus q} = \begin{pmatrix} p & 0 \\ 0 & 0 \end{pmatrix} + J_{p \oplus q} \] and \[ \begin{pmatrix} q & q \\ q & q \end{pmatrix} + J_{p \oplus q} = \begin{pmatrix} 0 & 0 \\ 0 & q \end{pmatrix} + J_{p \oplus q} \] where $q = e - p$. Consequently the map $\pi_p$ is unital.
\end{lem}

\begin{proof}
It suffices to show that \[ \pm \begin{pmatrix} 0 & p \\ p & p \end{pmatrix}, \pm \begin{pmatrix} q & q \\ q & 0 \end{pmatrix} \in C(p \oplus q). \] Let $\epsilon > 0$. Then \begin{eqnarray} \begin{pmatrix} 0 & p \\ p & p \end{pmatrix} + \epsilon \begin{pmatrix} p & 0 \\ 0 & q \end{pmatrix} + \frac{1}{\epsilon} \begin{pmatrix} q & 0 \\ 0 & p \end{pmatrix} & = & \begin{pmatrix} \epsilon p & p \\ p & (1+\frac{1}{\epsilon}) p \end{pmatrix} + \begin{pmatrix} \frac{1}{\epsilon} q & 0 \\ 0 & \epsilon q \end{pmatrix} \nonumber \\ & = & \begin{pmatrix} \epsilon & 1 \\ 1 & 1+\frac{1}{\epsilon} \end{pmatrix} \otimes p + \begin{pmatrix} \frac{1}{\epsilon} & 0 \\ 0 & \epsilon \end{pmatrix} \otimes q \geq 0. \nonumber \end{eqnarray} Also, \begin{eqnarray} \begin{pmatrix} 0 & -p \\ -p & -p \end{pmatrix} + \epsilon \begin{pmatrix} p & 0 \\ 0 & q \end{pmatrix} + (1+\frac{1}{\epsilon}) \begin{pmatrix} q & 0 \\ 0 & p \end{pmatrix} & = & \begin{pmatrix} \epsilon p & -p \\ -p & \frac{1}{\epsilon} p \end{pmatrix} + \begin{pmatrix} (1+\frac{1}{\epsilon}) q & 0 \\ 0 & \epsilon q \end{pmatrix} \nonumber \\ & = & \begin{pmatrix} \epsilon & -1 \\ -1 & \frac{1}{\epsilon} \end{pmatrix} \otimes p + \begin{pmatrix} 1 + \frac{1}{\epsilon} & 0 \\ 0 & \epsilon \end{pmatrix} \otimes q \geq 0. \nonumber \end{eqnarray} The proof that \[ \pm \begin{pmatrix} q & q \\ q & 0 \end{pmatrix} \in C(p \oplus q) \] is similar. For the final statement, we observe that \[ \pi_p(e) = \begin{pmatrix} e & e \\ e & e \end{pmatrix} + J_{p \oplus q} = \begin{pmatrix} p & p \\ p & p \end{pmatrix} + J_{p \oplus q} + \begin{pmatrix} q & q \\ q & q \end{pmatrix} + J_{p \oplus q} = \begin{pmatrix} p & 0 \\ 0 & q \end{pmatrix} + J_{p \oplus q}. \] Noting that $p \oplus q + J_{p \oplus q}$ is the unit of $M_2(\Cal V)/ J_{p \oplus q}$ concludes the proof \end{proof}

\begin{thm} \label{thm: projection representation}
Suppose that $\Cal V$ is an operator system and that $p \in \Cal V$ is an abstract projection. Then there exists a unital complete order embedding $\pi: \Cal V \rightarrow B(H)$ such that $\pi(p)$ is a projection in $B(H)$. 
\end{thm}

\begin{proof}
Since the abstract compression of $M_2(V)$ by $p \oplus q$ is an operator system, the direct sum \[ \rho = \bigoplus \phi, \] where the direct sum is taken over all ucp $\phi: M_2(\Cal V)/ J_{p \oplus q} \rightarrow M_n$ and all $n \in \mathbb{N}$, is a unital complete order embedding (see \cite{choi1977injectivity} as well as Chapter 13 of \cite{paulsen2002completely}). Replacing each $\phi$ with a corresponding $\psi$ as in Proposition \ref{prop: projection representation} we obtain a unital completely positive map $\rho'$ of $M_2(\Cal V)/ J_{p \oplus q}$ into $B(H)$ mapping $p \oplus 0 + J_{p \oplus q}$ to a projection. In fact, $\rho'$ is a complete order embedding. To see this, suppose that \[ \begin{pmatrix} a & b \\ b^* & c \end{pmatrix} + M_n(J_{p \oplus q}) \in M_n( M_2(\Cal V) / J_{p \oplus q}) \] is non-positive. Then it is necessarily the case that \[ \begin{pmatrix} a & 0 & 0 & b \\ 0 & 0 & 0 & 0 \\ 0 & 0 & 0 & 0 \\ b^* & 0 & 0 & c \end{pmatrix} + M_{2n}(J_{p \oplus q}) \] is also non-positive. It follows that there exists a unital completely positive map $\phi: M_2(\Cal V) / J_{p \oplus q} \to M_n$ such that \[ \phi_{2n} \left( \begin{pmatrix} a & 0 & 0 & b \\ 0 & 0 & 0 & 0 \\ 0 & 0 & 0 & 0 \\ b^* & 0 & 0 & c \end{pmatrix} + M_{2n}(J_{p \oplus q}) \right) \] is a non-positive matrix (for example, see the proof of Theorem 13.1 in \cite{paulsen2002completely}). It follows from Proposition \ref{prop: projection representation} that the corresponding map $\psi: M_2(\Cal V) / J_{p \oplus q} \to M_{n'}$ has the property that \[ \psi \begin{pmatrix} a & b \\ b^* & c \end{pmatrix} \] is non-positive. We conlcude that $\rho'$ is a unital complete order embedding. We complete the proof by precomposing $\rho'$ with the complete order embedding $\pi_p$ so that $\pi = \rho' \circ \pi_p$ is the desired unital complete order embedding.
\end{proof}

Theorem \ref{thm: projection representation} shows that when $p \in \Cal V$ is an abstract projection, we can build a complete order embedding of $\Cal V$ into $B(H)$ mapping $p$ to an ``honest'' projection. Of course a given operator system may contain many abstract projections, and a representation making $p$ into a projection may not map other abstract projections to projections. The next theorem shows that there is always one complete order embedding of $\Cal V$ which maps all abstract projections to concrete projections. The reader should compare the following with Theorem \ref{thm: Blecher-Neal}.

\begin{thm} \label{thm: projections in C*-envelope}
Let $\Cal V$ be an operator system, and suppose that $p \in \Cal V$ is an abstract projection. Then $p$ is a projection in its C*-envelope $C_e^*(\Cal V)$.
\end{thm}

\begin{proof}
Suppose that $p$ is an abstract projection in $\Cal V$, and let $j:\Cal V \to C_e^*(\Cal V)$ denote the inclusion map. By Theorem \ref{thm: projection representation}, there exists a unital complete order embedding $\phi: \Cal V \to B(H)$ with the property that $\phi(p)$ is a projection. Let $\Cal A := C^*(\phi(\Cal V))$. By the universal property of the C*-envelope, there exists a $*$-epimorphism $\pi: \Cal A \to C_e^*(\Cal V)$ satisfying $\pi(\phi(x)) = j(x)$ for all $x \in \Cal V$. Consequently \[ j(p) = \pi(\phi(p)) = \pi(\phi(p)^2) = \pi(\phi(p))^2 = j(p)^2. \] Since $j(p) = j(p)^*$, we conclude that $j(p)$ is a projection in $C_e^*(\Cal V)$.
\end{proof}

\section{Applications to quantum correlation sets}\label{sec: applications to qit}

We conclude with a brief application of our results to the theory of correlation sets in quantum information theory. We must first briefly recall some definitions.

Let $n,k \in \mathbb{N}$ be positive integers. We call a tuple $\{ p(a,b|x,y): a,b \in \{1,2,\dots,k\}, x,y \in \{1,2,\dots, n\}\}$ a \textit{correlation} if $p(a,b|x,y) \geq 0$ for each $a,b \leq k$ and $x,y \leq n$ and if, for each $x,y \leq n$, we have $ \sum_{a,b} p(a,b|x,y) = 1$. These conditions ensure that for each choice of $x$ and $y$, the matrix $\{p(a,b|x,y)\}_{a,b \leq k}$ constitutes a joint probability distribution. We say that a correlation $\{p(a,b|x,y)\}$ is \textit{non-signalling} if for each $x \leq n$ and $a \leq k$ the quantity \[ p_A(a|x) := \sum_b p(a,b|x,y) \]  is well-defined (i.e. independent of the choice of $y$), and that similarly for each $y \leq n$ and $b \leq k$ the quantity \[ p_B(b|y) := \sum_a p(a,b|x,y) \] is well-defined. We refer to the integer $n$ as the number of \textit{experiments} and the integer $k$ as the number of \textit{outcomes}. We let $C_{ns}(n,k)$ denote the set of non-signalling correlations with $n$ experiments and $k$ outcomes.

Correlations model a scenario where two parties, typically named Alice and Bob, are performing probabilistic experiments. Suppose Alice and Bob each have $n$ experiments, and that each experiment has $k$ possible outcomes. Then the quantity $p(a,b|x,y)$ denotes the probability that Alice performs experiment $x$ and obtains outcome $a$ while Bob performs experiment $y$ and obtains outcome $b$. Whenever Alice and Bob perform the experiments independently without communicating to one another, the resulting correlation is non-signalling. It is well-known (and easy to see) that the set $C_{ns}(n,k)$ of non-signalling correlations is a convex polytope when regarded as a subset of $\mathbb{R}^{n^2k^2}$ in the obvious way.

Let $H$ be a Hilbert space. We call a set $\{P_1, P_2, \dots, P_n\} \subset B(H)$ a \textit{projection-valued measure} if each $P_i$ is a projection on $H$ and $\sum_i P_i = I$. A correlation $\{p(a,b|x,y)\} \in C_{ns}(n,k)$ is called a \textit{quantum commuting} correlation if there exists a Hilbert space $H$, a unit vector $\phi \in H$, and projection valued measures $\{E_{x,a}\}_{a=1}^k, \{F_{y,b}\}_{b=1}^k \subset B(H)$ for each $x,y \leq n$ satisfying the conditions that $E_{x,a}F_{y,b} = F_{y,b}E_{x,a}$ for all $x,y \leq n$ and $a,b \leq k$ and \[ p(a,b|x,y) = \innerproduct{\phi}{E_{x,a}F_{y,b}\phi}. \] The set of all quantum commuting correlations with $n$ experiments and $k$ outcomes is denoted by $C_{qc}(n,k)$. It is well-known that $C_{qc}(n,k)$ is a closed convex subset of $C_{ns}(n,k)$ and that it is not a polytope for any $n \geq 2$ or $k \geq 2$. 

If we modify the definition of the quantum commuting correlations by requiring the Hilbert space $H$ to be finite dimensional, we obtain a \textit{quantum} correlation. The set of quantum correlations with $n$ experiments and $k$ outcomes are denoted by $C_q(n,k)$. The set $C_q(n,k)$ is known to be a convex subset of $C_{qc}(n,k)$. It was shown by William Slofstra in \cite{slofstra2019set} that for some values of $n$ and $k$, $C_q(n,k)$ is non-closed, and hence $C_q(n,k)$ is a proper subset of $C_{qc}(n,k)$. The recent preprint \cite{ji2020mip} shows that for some ordered pair $(n,k)$ the closure of $C_q(n,k)$ is a proper subset of $C_{qc}(n,k)$. Precisely which ordered pairs $(n,k)$ satisfy this relation remains unknown.

One reason questions about $C_q(n,k)$ and $C_{qc}(n,k)$ are difficult to answer is that these sets are defined by applying arbitrary vector states to arbitrary projection-valued measures acting on arbitrary Hilbert spaces. In principle, it may be easier to understand the sets $C_q(n,k)$ and $C_{qc}(n,k)$ if there were an equivalent definition which was independent of Hilbert spaces and Hilbert space operators. The following proposition indicates that such a characterization is, in some sense, possible.

\begin{prop} \label{prop: equivalent to correlation}
Let $n$ and $k$ be positive integers. Then the following statements are equivalent.
\begin{enumerate}
    \item $\{p(a,b|x,y)\} \in C_{qc}(n,k)$ (resp. $\{p(a,b|x,y)\} \in C_{q}(n,k)$).
    \item There exists a (resp. finite dimensional) C*-algebra $\Cal A$, projection valued measures $\{E_{x,a}\}_{a=1}^k, \{F_{y,b}\}_{b=1}^k \subset \Cal A$ for each $x,y \leq n$ satisfying $E_{x,a}F_{y,b} = F_{y,b}E_{x,a}$ for all $x,y \leq n$ and $a,b \leq k$, and a state $\phi: \Cal A \to \mathbb{C}$ such that $p(a,b|x,y) = \phi(E_{x,a}F_{y,b})$.
    \item There exists an operator system $\Cal V \subset B(H)$ (resp. for a finite dimensional Hilbert space $H$), projection valued measures $\{E_{x,a}\}_{a=1}^k, \{F_{y,b}\}_{b=1}^k$ for each $x,y \leq n$ satisfying $E_{x,a}F_{y,b} \in \Cal V$ and $E_{x,a}F_{y,b} = F_{y,b}E_{x,a}$ for all $x,y \leq n$ and $a,b \leq k$, and a state $\phi: \Cal V \to \mathbb{C}$ such that $p(a,b|x,y) = \phi(E_{x,a}F_{y,b})$.
\end{enumerate}
\end{prop}

\begin{proof}
The equivalence of 1 and 3 are obvious, taking $\Cal V$ to be the linear span of the operator products $E_{x,a}F_{y,b}$. The equivalence of 1 and 2 is an application of the GNS construction, taking $\Cal A$ to be the C*-algebra generated by the set $\{E_{x,a}F_{y,b}\}$.
\end{proof}

In principle, statement 2 of Proposition \ref{prop: equivalent to correlation} provides an abstract characterization of correlation sets in that it is independent of Hilbert space representation. However, C*-algebras are themselves complex structures, so it is not clear that statement 2 of Proposition \ref{prop: equivalent to correlation} is a significant improvement over the definitions of $C_{qc}(n,k)$ and $C_q(n,k)$. Moreover, the correlation is generated by applying a state on the C*-algebra $\Cal A$ to operators of the form $E_{x,a}F_{y,b}$, which span only a linear subspace of $\Cal A$. Thus it seems that one can get by with significantly less data than the C*-algebra $\Cal A$ has to offer. These observations make statement 3 of Proposition \ref{prop: equivalent to correlation} seem more appealing, except that we have insisted that the operator system $\Cal V$ be concretely represented so that we can enforce the relations $E_{x,a}F_{y,b} = F_{y,b}E_{x,a}$ and that each $E_{x,a}$ and $F_{y,b}$ are projections. Theorem \ref{thm: projections in C*-envelope} provides us with the tools to ensure that $E_{x,a}$ and $F_{y,b}$ are projections in an abstract operator system. It remains to show that the condition $E_{x,a}F_{y,b} = F_{y,b}E_{x,a}$ can also be enforced in an abstract operator system.

\begin{defn} \label{defn: ns and qc operator systems}
Let $n,k \in \mathbb{N}$. We call an operator system $\Cal V$ a \textit{non-signalling} operator system if it is the linear span of positive operators $\{Q(a,b|x,y) : a,b \leq k, x,y \leq n\} \subset \Cal V$, called the \textit{generators} of $\Cal V$, with the properties that $\sum_{a,b} Q(a,b|x,y) = e$ for each choice of $x,y \leq n$ and that the operators \[ E(a|x) := \sum_b Q(a,b|x,y) \] and \[ F(b|y) := \sum_a Q(a,b|x,y) \] are well-defined (i.e. $E(a|x)$ is independent of the choice of $y$ and $F(b|y)$ is independent to the choice of $x$). We call an operator system $\Cal V$ a \textit{quantum commuting} operator system if it is a non-signalling operator system with the property that each generator $Q(a,b|x,y)$ is an abstract projection in $\Cal V$.
\end{defn}

The next theorem justifies the choice of terminology in Definition \ref{defn: ns and qc operator systems}.

\begin{thm} \label{thm: characterization correlations}
A correlation $\{p(a,b|x,y)\}$ is non-signalling (resp. quantum commuting) if and only if there exists a non-signalling (resp. quantum commuting) operator system $\Cal V$ with generators $\{Q(a,b|x,y)\}$ and a state $\phi: \Cal V \to \mathbb{C}$ such that $p(a,b|x,y) = \phi(Q(a,b|x,y))$ for each $a,b,x,y$.
\end{thm}

\begin{proof}
We first verify the equivalence for non-signalling correlations. Suppose that $\Cal V$ is a non-signalling operator system with generators $Q(a,b|x,y)$. Let $\phi: \Cal V \to \mathbb{C}$ be a state, and define $p(a,b|x,y) := \phi(Q(a,b|x,y))$. Then for each $x,y,a,b$ we have $\phi(Q(a,b|x,y) \geq 0$ since $\phi$ is positive, and for each $x$ and $y$ we have \[ \sum_{a,b} \phi(Q(a,b|x,y)) = \phi(\sum_{a,b} Q(a,b|x,y)) = \phi(e) = 1. \] So $\{p(a,b|x,y)\}$ is a correlation. Similarly, the quantities $p_A(a|x)$ and $p_B(b|y)$ are well-defined for each $a,b,x,y$ with $p_A(a|x) = \phi(E(a|x))$ and $p_B(b|y) = \phi(F(b|y))$. So $\{p(a,b|x,y)\}$ is a non-signalling correlation.

On the other hand, suppose that $\{p(a,b|x,y)\}$ is a non-signalling correlation. Let $H = \mathbb{C}$ regarded as a one-dimensional Hilbert space. Set $Q(a,b|x,y) = p(a,b|x,y)$ for each $a,b,x,y$. Clearly if $\Cal V = \Span \{Q(a,b|x,y)\} = \mathbb{C} = B(H)$ then $\Cal V$ is a non-signalling operator system with generators $\{Q(a,b|x,y)\}$, since $p(a,b|x,y)$ is a non-signalling correlation. Let $\phi: \Cal V \to \mathbb{C}$ be the state $\phi(\lambda) = \lambda$. Then $\phi(Q(a,b|x,y)) = p(a,b|x,y)$ for each $a,b,x,y$.

We now consider the equivalence for quantum commuting correlations. The forward direction is immediate from Proposition \ref{prop: equivalent to correlation}. We show the converse. Suppose that $\Cal V$ is a quantum commuting operator system with generators $\{Q(a,b|x,y)\}$, and let $\phi: \Cal V \to \mathbb{C}$ be a state. Since $\Cal V$ is quantum commuting, each $Q(a,b|x,y)$ is an abstract projection. Therefore by Theorem \ref{thm: projections in C*-envelope} each $Q(a,b|x,y)$ is a projection when regarded as an element of $\Cal A := C_e^*(\Cal V)$. Since $\sum_{a,b} Q(a,b|x,y) = e$, it follows that $Q(a,b|x,y)Q(c,d|x,y) = 0$ in $\Cal A$ whenever $a \neq c$ or $b \neq d$. Hence $E(a|x)$ and $F(b|y)$ are also projections in $\Cal A$. Moreover, \[ E(a|x)F(b|y) = (\sum_d Q(a,d|x,y))(\sum_c Q(c,b|x,y)) = Q(a,b|x,y) \] since $Q(a,d|x,y)Q(c,b|x,y)=0$ whenever $a \neq c$ or $b \neq d$. Finally, by the Arveson Extension Theorem, there exists a state $\tilde{\phi}: \Cal A \to \mathbb{C}$ which extends $\phi$, since $\Cal V$ is a unital self-adjoint subspace of $\Cal A$. By part 2 of Proposition \ref{prop: equivalent to correlation}, $p(a,b|x,y) = \tilde{\phi}(E(a|x)F(b|y))$ is a quantum commuting correlation. But $\tilde{\phi}(E(a|x)F(b|y)) = \phi(Q(a,b|x,y))$, so the proof is complete.
\end{proof}

\section*{Acknowledgments}

The authors would like to thank Ivan Todorov for pointing us towards the reference \cite{lupini2020perfect}, Brent Nelson and Thomas Sinclair for helpful comments on an earlier version of this manuscript, and the referee for their careful reading of the manuscript. This research was initiated at the conference ``QLA meets QIT," which was funded by the NSF grant DMS-1600857 and the Yue Lin Lawrence Tong Endowment Fund, Purdue University. The first author was supported by a Purdue Research Foundation Grant from the Department of Mathematics, Purdue University.

\bibliographystyle{plain}
\bibliography{references}
\end{document}